\documentclass[a4paper]{article}

\usepackage[english]{babel}

\usepackage[utf8]{inputenc}
\usepackage[utf8]{inputenc}
\usepackage{amsmath}
\usepackage{graphicx}
\usepackage{amssymb}
\usepackage{amsthm}
\usepackage{tikz-cd}
\usepackage{mathrsfs}
\usepackage[colorinlistoftodos]{todonotes}
\usepackage{enumitem}
\usepackage{yfonts}
\usepackage{ dsfont }
\usepackage{MnSymbol}
\usepackage{slashed}

\title{Degeneration of Calabi-Yau metrics and canonical basis}

\author{Yang Li}

\date{\today}
\newtheorem{thm}{Theorem}[section]
\newtheorem{lem}[thm]{Lemma}

\theoremstyle{definition}
\newtheorem{eg}[thm]{Example}

\newtheorem{conj}[thm]{Conjecture}
\newtheorem{cor}[thm]{Corollary}

\newtheorem{rmk}[thm]{Remark}
\newtheorem{prop}[thm]{Proposition}
\newtheorem{Def}[thm]{Definition}
\newtheorem*{Question}{Question}

\newtheorem*{Acknowledgement}{Acknowledgement}

\newcommand{\cf}{\emph{cf.} }

\newcommand{\R}{\mathbb{R}}
\newcommand{\C}{\mathbb{C}}
\newcommand{\Z}{\mathbb{Z}}
\newcommand{\N}{\mathbb{N}}
\newcommand{\Q}{\mathbb{Q}}

\newcommand{\norm}[1]{\left\lVert#1\right\rVert}

\def\Xint#1{\mathchoice
	{\XXint\displaystyle\textstyle{#1}}%
	{\XXint\textstyle\scriptstyle{#1}}%
	{\XXint\scriptstyle\scriptscriptstyle{#1}}%
	{\XXint\scriptscriptstyle\scriptscriptstyle{#1}}%
	\!\int}
\def\XXint#1#2#3{{\setbox0=\hbox{$#1{#2#3}{\int}$ }
		\vcenter{\hbox{$#2#3$ }}\kern-.6\wd0}}

\def\dashint{\Xint-}

\begin{document}
	\maketitle

	\begin{abstract}
		For polarised degenerations of Calabi-Yau manifolds whose essential skeleton has dimension $1\leq m\leq n$, we show that  the $C^0$ potential theoretic limit of the Calabi-Yau metrics agrees with the non-archimedean Calabi-Yau metric on the Berkovich analytification. Moreover, this limit data can be encoded into the unique minimiser of the Kontorovich functional of an optimal transport problem, under some algebro-geometric assumptions on the existence of a canonical basis of sections for tensor powers of the polarisation line bundle.
	\end{abstract}

	\section{Introduction}

	Let $X\to \mathbb{D}^*_t$ be a \emph{polarised meromorphic degeneration} of Calabi-Yau (CY) manifolds over the punctured disc, namely a 1-parameter complex analytic family of $n$-dimensional projective CY manifolds $X\to \mathbb{D}^*$ defined over the field of convergent Laurent series, equipped with a relatively ample line bundle $L\to X$.  We assume there is a nowhere vanishing holomorphic volume form $\Omega\in H^0(X, K_X)$, which induces the fibrewise holomorphic volume form $\Omega_t$ by $\Omega=\Omega_t\wedge dt$.

	A basic question is to understand the limiting behaviour of the CY metrics $(X_t, \omega_{CY,t})$ in the class $c_1(L)$, as $t\to 0$. There is a sharp dichotomy depending on the dimension $m$ of the essential skeleton $Sk(X)$ attached to the polarised degeneration:
	
	\begin{itemize}
		\item (Algebraic case) If $m=0$, then the metrics $(X_t, \omega_{CY,t})$  are uniformly volume non-collapsing, and converge in the Gromov-Hausdorff sense to a singular CY metric on a CY variety with klt singularity \cite{DonaldsonSun}\cite{EGZ}. 
		
		\item (Transcendental case) If $m\geq 1$, then the CY metrics exhibit volume collapsing, and it is expected that the Gromov-Hausdorff limit has real dimension $m$. This case will be the focus of this paper. The special case $m=n$ is called the large complex structure limit, which is the context of the metric SYZ conjecture \cite{SYZ}.
	\end{itemize}

	The picture that emerged over the years, is that for $m\geq 1$, \emph{the metric limit should be captured by non-archimedean (NA) geometry}. In more detail, the meromorphic degeneration family induces by base change a family $X_K\to \text{Spec}(K)$ for $K=\C(\!(t)\!)$, and one can associate the Berkovich space $X_K^{an}$. NA pluripotential theory \cite{Boucksom} provides a unique up to scale semipositive metric $\norm{\cdot}_{CY,0}$ on the line bundle $L\to X_K^{an}$, whose NA Monge-Amp\`ere measure equals the Lebesgue measure supported on the essential skeleton $Sk(X)$ with total integral $(L^n)$. There is a hybrid topology on $X_K^{an}\cup \bigcup_{t\neq 0} X_t$ unifying the Berkovich space with the CY manifolds. Let $h_{CY,t}$ denote the Hermitian metric on $L\to X_t$ whose Chern curvature induce the CY metrics $(X_t, \omega_{CY,t})$ in the class $c_1(L)$. The general expectation is the following (\cf discussions after \cite[Thm C]{Boucksom1}):

	\begin{conj}\label{conj:hybridconvergence}
		Let $m\geq 1$. After suitable normalisation, the norm function $|\cdot |_{h_{CY,t}}^{\frac{1}{|\log |t||}}$ on $L\to X_t$ converges in the $C^0$-hybrid topology to $\norm{\cdot}_{CY,0}$ on $L\to X_K^{an}$ as $t\to 0$.
	\end{conj}

	We now summarize the main known approaches to study the limit of the CY metrics as $t\to 0$. Except for the few special cases (eg. the `small complex structure limit' \cite{SunZhang}\cite{HSVZ}, and the elliptic K3 example \cite{GrossWilson}) where one can perform a refined gluing construction, all the other methods rely on complex pluripotential theory, which applies in general, plus some additional information about the polarised degeneration family in question, which is method-specific.
	
	\begin{itemize}

		\item  In the case of the Fermat family of hypersurfaces \cite{LiFermat}\cite{Hultgren}, one crucially relies on the limiting description of K\"ahler potentials invariant under a large discrete symmetry group, in terms of \emph{tropical  combinatorics}. This approach is very explicit, but its main difficulty is that the combinatorics for most other examples are highly complicated.

		\item  In some situations one can formulate an \emph{optimal transport} problem, whose optimizer solves a real Monge-Amp\`ere type equation in favourable cases. Then one can show that this solution captures the $C^0$-limit of the CY potentials on $X_t$ as $t\to 0$. This approach is partially successful for certain families of toric Fano hypersurfaces \cite{LiFano}\cite{Hultgren}\cite{Hultgren2}\cite{Goto}, and for the `intermediate complex structure limit' examples \cite{Liintermediate}. In contrast, it is also known that in many toric Fano hypersurface examples, the na\"ive formulation of this optimal transport problem \emph{does not} produce the desired real Monge-Amp\`ere solution \cite{Hultgren2}, as optimal transport between polyhedral complexes is not always as well behaved as optimal transport between domains in Euclidean spaces. In this paper we will suggest that the cost function in the na\"ive formulation of the optimal transport problem is only the first approximation to the correct cost function.

		\item  In the case of large complex structure limit, one can conditionally prove enough about the limiting behaviour of the CY potential to deduce the metric SYZ conjecture, by assuming some conjectural `comparison property' about the solution of the \emph{non-archimedean Monge-Amp\`ere (NA MA) equation} \cite{LiNA}.  The known cases of this conjectural `comparison property' essentially coincides with the cases where either of the above two methods work.
		
	\end{itemize}

	Our first main result is

	\begin{thm}\label{thm:hybridconvergence}
		Conjecture \ref{conj:hybridconvergence} holds.
	\end{thm}

	\begin{rmk}
		Theorem \ref{thm:hybridconvergence} is a direct generalisation of \cite[Theorem 1.1]{Liintermediate}, which concerns the special example of the `intermediate complex structure limit'. The proof of \cite[Theorem 1.1]{Liintermediate} relies on the fact that the essential skeleton $Sk(X)$ is a \emph{simplex}, which is a domain in the Euclidean space, so one can solve a suitable optimal transport problem classically, and use the solution to produce a generalised Calabi ansatz metric on $X_t$ which is approximately Calabi-Yau. This setting precludes the large complex structure limit case, where $Sk(X)$ is known to be a pseudomanifold with the same rational homology as the $n$-sphere, when $X_t$ is  strict Calabi-Yau \cite{NicaiseXu}. 
		In our more general setting we allow for the large complex structure limit.
	\end{rmk}

	Our second main goal is to find a more \emph{concrete} description for the NA CY metric $\norm{\cdot}_{CY,0}$ on the Berkovich space, in the case of the large complex structure limit, under some additional algebro-geometric assumptions
	inspired by the theta functions in the Gross-Siebert programme.
	The key property is

	%Our first goal is to formulate a more algebro-geometric sufficient condition for Conjecture \ref{conj:hybridconvergence}; this condition is inspired by the \emph{canonical basis} of theta functions in mirror symmetry. To simplify the matter, we will assume that $(X,L)$ admits a \emph{semistable SNC model} over $S$, which can be always achieved after finite base change by Hironaka resolution and the semistable reduction theorem; the main benefit is that we can then work with Laurent polynomials instead of Puiseaux series. To make sense of the following Definition, the reader should recall that the essential skeleton $Sk(X)$ embeds canonically into the Berkovich space, hence the points on $Sk(X)$ have the intrinsic meaning as monomial valuations (See Section \ref{sect:valuativeindependence} for more explanations). 

	\begin{Def}\label{Def:valuativeindependence}
		For every $l\geq 1$, let $N(l)= \dim H^0(X_t, lL)$, and let  $\theta_1^l ,\ldots \theta^l_{N(l)}$ be holomorphic sections of $L^{\otimes l}\to X$, which are meromorphic at $0\in \mathbb{D}_t$. We say that $\theta_1^l,\ldots \theta^l_{N(l)}$ satisfy the \emph{valuative independence condition}, if for every $x\in Sk(X)$, and every coefficient function $a_i(t)\in K=\C(\!(t)\!)$, we have 
		\[
		val_x(\sum_i a_i \theta^l_i)= \min_{i: a_i\neq 0}  (val(a_i) + val_x(\theta^l_i) ).
		\]

	\end{Def}

	\begin{rmk}(\cf section \ref{sect:canonicalbasis})
		The valuative independence condition implies that the sections $\theta_1^l,\ldots \theta_{N(l)}^l$ are linearly independent over the field $K$. In particular, for small $t\in \mathbb{D}^*$, the restrictions of $\theta_i^l$ for $i=1,\ldots ,N(l)$ are linearly independent in $H^0(X_t, lL)$, hence form a distinguished \emph{basis} of this space of sections.

		For comparison, in Gross-Siebert mirror symmetry, the mirror family of Calabi-Yau manifold carries a natural polarisation line bundle, and the space of sections for tensor powers of this line bundle carries a \emph{canonical basis} known as \emph{theta functions}. We think it is an important question \emph{whether the theta functions satisfy the valuative independence condition}. %This question is mostly unknown for polarised degeneration of compact Calabi-Yau manifolds, but in the related context of open Calabi-Yau manifolds arising from cluster varieties, a version of the valuative independence for theta functions has been observed in the exciting recent work of Cheung-Magee-Mandel-Muller \cite{Mandel}. 
	\end{rmk}

	As the statement is long, we  only summarize the main conceptual features here:
	
	\begin{itemize}
		
		\item We will assume that there is a distinguished basis 
		of sections $\theta_i^l$ satisfying the valuative independence condition. We assume $\theta_i^l$ are parametrised by suitable rational points on some integral polyhedral complex $B$, and they satisfy some ring structure properties (See Section \ref{sect:assumptions}). (This is motivated by Gross-Siebert mirror symmetry, where $B$ is the essential skeleton of a polarised degeneration of CY manifolds mirror to $X$.)

		\item  The NA CY metric $\norm{\cdot}_{CY,0}$ on the Berkovich space $X_K^{an}$ is determined by its restriction to the essential skeleton $Sk(X)\subset X_K^{an}$. (This is a direct consequence of the domination principle \cite[Lemma 8.4]{Boucksom}, and the fact that the NA MA measure is supported on $Sk(X)$.)

		\item The limiting behaviour for the valuation of the basis sections $\theta_i^l$ determines a \emph{cost function} $c(x,p)$ on $Sk(X)\times B$.

		\item  The potential for the NA CY metric restricted to $Sk(X)$,  is the minimiser for the \emph{Kontorovich dual formulation of an optimal transport problem} between $Sk(X)$ and $B$. 
	\end{itemize}

	The method here works also for some polarized degenerations with $1\leq m=\dim Sk(X)<n$. In the case of the intermediate structure limit examples as in \cite{Liintermediate}, we will give a new proof for why the potential function of  $\norm{\cdot}_{CY,0}$ agrees with the optimizer of an optimal transport problem.

	\textbf{Organization}. In Section \ref{sect:recapNA} we briefly recall some basic facts about non-archimedean pluripotential theory, with some emphasis on the hybrid topology convergence result of Favre in Section \ref{NAMAcxMA}, and the formula for the relative Monge-Amp\`ere energy by Boucksom-Eriksson in Section \ref{sect:relativevolume}. Section \ref{sect:canonicalbasis} explains the valuative independence condition in Def. \ref{Def:valuativeindependence} and a more localized variant \ref{def:valuativeindependenceatx},  exhibits a number of examples, and discusses some motivation from theta functions in mirror symmetry.  Section \ref{sect:potentialconvergence} proves the hybrid convergence Thm. \ref{thm:hybridconvergence} by adapting the Bergman kernel estimate technique in the intermediate complex structure case \cite{Liintermediate}, along with some new ingredients. It turns out that a weaker substitute of a valuative independent basis always exists, and is already sufficient for Thm. \ref{thm:hybridconvergence}.
	Section \ref{sect:optimaltransport} shows under a number of extra assumptions that the NA CY potential is equal to the unique minimiser of the Kontorovich functional for some optimal transport problem. This is inspired by the Chebyshev transform of Witt-Nystr\"om \cite{WittNystrom}, and builds on Boucksom-Eriksson's formula for the relative Monge-Amp\`ere energy. Some connections to the literature are discussed in Section \ref{sect:discussion}.

	\begin{Acknowledgement}
		The author is sponsored by the Royal Society University Research Fellowship. He thanks Mark Gross, Pierrick Bousseau and Travis Mandel for discussions on theta functions in mirror symmetry.
	\end{Acknowledgement}

	\section{Recap: non-archimedean geometry}\label{sect:recapNA}

	We recall a few basic notions from non-archimedean geometry, based on the work of Boucksom et al. \cite{Boucksom}\cite{Boucksom1}\cite{Boucksomnew1}\cite{Boucksomsemipositive}\cite{Boucksomsurvey}; see also \cite[section 5]{Boucksom1}\cite[section 2]{Liintermediate}.

	\subsection{Models and dual complex}

	We denote $R=\C[\![ t ]\!]$, and $K=\C(\!(t)\!)$ with the standard discrete valuation. Given a polarised degeneration family  $X\to S\setminus \{ 0\}$, we let $X_K$ be a smooth, connected, projective variety $X_K$ of dimension $n$ over the formal punctured disc $\text{Spec}(K)$. Let $X_K^{an}$ denote the  \emph{Berkovich space}. Given an affine variety over $K$, the Berkovich space is the space of semivaluations on the ring of functions extending the standard discrete valuation on $K$, and in general $X_K^{an}$ is obtained by gluing the affine pieces.	The semivaluations $v$ are equivalently thought of as multiplicative seminorms $|\cdot |=e^{-v}$, satisfying the ultrametric property
	\[
	|f+g|\leq \max\{ |f|+|g|  \}, \quad |fg|=|f||g|.
	\]

	A \emph{model} of $X_K$ is a normal projective scheme $\mathcal{X}$, flat and of finite type over $\text{Spec} (R)  $, together with an identification of the generic fibre of the morphism $\mathcal{X}\to\text{Spec} (R)  $ with $X_K$.  For any model $\mathcal{X}$ and every irreducible component $E$ of $\mathcal{X}_0$, we write the central fibre as $\mathcal{X}_0=\sum b_iE_i$, and $E_J= \cap_{i\in J}E_i$. We say $\mathcal{X}$ is SNC if $\mathcal{X}$ is regular and $\mathcal{X}_0$ has SNC support, and an SNC model is semistable if all the divisors components on $\mathcal{X}_0$ have multiplicity one. By Hironaka resolution and semistable reduction, up to finite base change of $S$, we shall without loss assume that there exists some semistable SNC model over $S$. A model of the line bundle $L\to X_K$ is a $\Q$-line bundle $\mathcal{L}$ on a proper model $\mathcal{X}$, together with an identification $\mathcal{L}|_{X_K}=L$.

	The \emph{dual complex} $\Delta_{\mathcal{X}}$ of an SNC model $\mathcal{X}$ is the simplicial complex whose vertices correspond to the irreducible components $E_i$ of the central fibre, and whose faces correspond to the connected components of nonempty intersection strata $E_J$, and are identified as the simplex 
	$\Delta_J\simeq \{  x\in \R_{\geq 0}^{|J|}| \sum_{i\in J} b_ix_i=1 \}.
	$

	Given an SNC model,
	there is a natural \emph{embedding map} of the dual complex into $X_K^{an}$, which is a homeomorphism onto its image. Given $x\in \Delta_J$, we need to associate a quasi-monomial valuation $val_x$, regarded as a point in $X_K^{an}$. Given any local function $f\in \mathcal{O}_{\mathcal{X},\eta}$ where $\eta$ is the generic point of $E_J$, we can Taylor expand
	\[
	f=\sum_{\alpha\in \N^{|J|}} f_\alpha z_0^{\alpha_0}\ldots z_{|J|}^{\alpha_{|J|}},
	\] 
	where $z_i$ are local equations of the divisors $E_i$, and $f_\alpha$ is either a unit or zero. Then $val_x$ is defined as the minimal weighted vanishing order:
	\[
	val_x(f)= \min\{   \langle x, \alpha\rangle  |  f_\alpha\neq 0  \}.
	\]
	In particular, the vertices of $\Delta_J$ map to the divisorial valuations. The requirement $\sum b_ix_i=1$ comes from $val(t)=1$, where $t\in R$ is the uniformizer. Henceforth we identify the dual complex with its image in $X_K^{an}$.

	\begin{rmk}
		For many purposes, SNC models can be replaced by dlt (divisorially log terminal) models, which arises naturally in the minimal model program \cite{NicaiseXu}. The dual complex of a dlt model is simply defined as the dual complex of its SNC locus. We will only make use of dlt models in the intermediate complex structure example as in \cite{Liintermediate}. 
	\end{rmk}

	Moreover, there is a natural topology on $X_K^{an}\cup X$, known as the \emph{hybrid topology}, which describes $X_K^{an}$ as a limiting object of the complex manifolds $X_t$ \cite[Appendix]{Boucksom1}\cite[Section 1.2]{PilleSchneider}; a one-page recap is in \cite[section 2.2]{Liintermediate}. We will mention some key facts below.

	\subsection{Essential skeleton and volume form asymptote}\label{section:essentialskeleton}

	Suppose $X$ is a meromorphic degeneration family of smooth projective CY manifolds $X_t$ over some punctured disc, and $X$ is equipped with a holomorphic volume form $\Omega$, which induces the holomorphic volume forms $\Omega_t$ on the $X_t$ via $\Omega= \Omega_t\wedge dt$. The \emph{essential skeleton} is the subset of $X_K^{an}$ where the log discrepancy function takes the minimal value $\kappa_0$ \cite[Section 3.1.1]{NicaiseXu}.  Up to passing to semistable reduction and multiplying $\Omega$ by a suitable power of $t$, we may assume $\kappa_0=0$.

	We follow \cite{Boucksom1} to sketch the computation of the CY volume form asymptote. Let $\mathcal{X}$ be an SNC model of the degeneration family, whose dual complex is $\Delta_{\mathcal{X}}$. We write the central fibre as $\mathcal{X}_0=\sum b_iE_i$, and the relative log canonical divisor as $\sum a_i E_i$ (so the log discrepancy function takes the value $\frac{a_i}{b_i}$ at the divisorial point corresponding to $E_i$). We consider an intersection stratum $E_J= \cap_{i\in J} E_i$, and let $z_0,\ldots z_n$ be local coordinates on $E_J$ bounded away from the deeper intersection strata, such that $z_0,\ldots z_k$ are the local defining equation for the $E_i$ with $i\in J$. Locally
	\[
	\Omega=f z_0^{a_0+b_0-1}\ldots z_k^{a_k+b_k-1} dz_0\wedge \ldots  dz_n,
	\]
	and we can arrange
	$
	t= z_0^{b_0}\ldots z_k^{b_k} .
	$
	Then
	\begin{equation}\label{holovollocal}
		\Omega_t= b_0^{-1} f z_0^{a_0}\ldots z_k^{a_k} d\log z_1\wedge \ldots d\log z_k\wedge dz_{k+1}\wedge \ldots dz_n.
	\end{equation}
	By the non-negativity of the log discrepancy function, we have $a_i\geq 0$ for $i=0,\ldots k$.

	If at least one exponent $a_i>0$, then the local contribution to the volume integral $\int_{X_t} \Omega_t\wedge \overline{\Omega}_t$ is suppressed by a factor $O(\frac{1}{|\log |t||})$. The divisors with $a_i=0$ correspond to the vertices of a subcomplex of $\Delta_{\mathcal{X}}$, which is naturally identified with the essential skeleton $Sk(X)\subset X_K^{an}$ under the embedding map.  The dimension $m=\dim Sk(X)$ then equals the largest value of $k$ such that $a_0=\ldots =a_k=0$, so the total volume integral $\int_{X_t} \Omega_t\wedge \overline{\Omega}_t=O(|\log |t||^m)$. We denote the normalized CY measure as
	\[
	\mu_t= \frac{  \Omega_t\wedge \overline{\Omega}_t }{   \int_{X_t} \Omega_t\wedge \overline{\Omega}_t      } ,
	\]
	which is a probability measure on $X_t$.
	Boucksom-Jonsson \cite{Boucksom} show that

	\begin{thm}\label{Volumeconvergence}
		Under the hybrid topology, the measures $\mu_t$ on $X_t$ converge weakly to a probability measure $\mu_0$ supported on $Sk(X)\subset X_K^{an}$, which puts no measure on the lower dimensional faces of $Sk(X)$, and agrees with a suitably normalized Lebesgue measure $\mu_0$ on the $m$-dimensional faces of $Sk(X)$.
		More concretely, for any continuous function $f$ on the hybrid space $X_K^{an}\cup X$, we have
		\[
		\int_{X_t} fd\mu_t\to \int_{X_K^{an}} fd\mu_0.
		\]
	\end{thm}

	The normalisation constants can be pinned down by a more invariant interpretation of the local computation (\ref{holovollocal}). Let $E_J=\cap_0^m E_i$ correspond to an $m$-dimensional face $\Delta_J\subset Sk(X)$. The formula (\ref{holovollocal}) can be written as
	\[
	\Omega_t= \Omega_{E_J}\wedge d\log z_1\wedge\ldots d\log z_m,
	\]
	where $\Omega_{E_J}$ extends holomorphically over $E_J$ (since all $a_i=0$ here), and defines a holomorphic volume form on $E_J$ called the Poincar\'e residue, which is independent of the choice of local coordinates. The integral
	$
	\int_{E_J} \Omega_{E_J}\wedge \overline{\Omega}_{E_J}
	$
	is finite, and 
	\begin{equation}\label{Lebesguemeasure}
		\mu_0= (C_0 \int_{E_J} \Omega_{E_J}\wedge \overline{\Omega}_{E_J}) |dx_1\ldots dx_m|
	\end{equation}
	as a measure on the face $\Delta_J\simeq \{  x\in \R_{\geq 0}^{m+1}|\sum_0^m b_i x_i=1  \}$, 
	where the constant $C_0$ is independent of the face, and only serves to normalise $\mu_0$ to a probability measure.

	\subsection{Semipositive metrics, NA Calabi conjecture}

	A \emph{continuous metric} $\norm{\cdot}$ on the line bundle $L\to X_K^{an}$ associates to local sections $s$ a continuous function $\norm{s}$, compatible with the restriction of sections, and such that $\norm{fs}=|f|\norm{s}$ for local functions $f$. A choice of model line bundle $\mathcal{L}\to \mathcal{X}$ induces a `model metric' on $L$, with the property that the locally invertible sections of $\mathcal{L}$ have norm one. We can then write any other continuous metric in terms of a potential function, via $\norm{\cdot}=\norm{\cdot}_{\mathcal{L}}e^{-\phi}$, and it is conventional to identify the metric with the potential.

	Let $L$ be a relatively ample line bundle. Given a NA norm on the $K$-vector space $V=H^0(X_K,lL)$ which generates the line bundle $L^{\otimes l}\to X$, the \emph{NA Fubini-Study metric} on the line bundle $L\to X_K^{an}$ is defined by
	\[
	\norm{s}_{FS,l}(x)=  \inf_{ \tilde{s}\in V, \tilde{s}(x)=s^{\otimes l}(x)} \norm{ \tilde{s}}_V^{1/l}, \quad \forall x\in X_K^{an}.
	\]
	Concretely, given an orthogonal $K$-basis $s_1, \ldots s_N$ for $V$, namely that
	\[
	\norm{ a_1s_1+ \ldots+ a_N s_N}_V=\max\{ |a_1|\norm{s_1}_V,\ldots ,|a_N|  \norm{s_N}_V  \}, \quad \forall a_i\in K,
	\]
	and writing  $\norm{s_j}_V= e^{c_j} $ for $c_j\in \R$,   then
	\begin{equation}\label{FubiniStudyNA}
		\norm{s}_{FS ,l}(x)= \frac{ |s(x)|}{ \max_j \{  |s_j(x)|e^{-c_j}  \}^{1/l}   }, \quad \forall x\in X^{an}.
	\end{equation}
	This can be recast in terms of potentials
	\begin{equation}\label{NAFubiniStudy}
		\phi_{FS,l}= \frac{1}{l} \max_{1\leq i\leq N}(\log |s_i| -c_i),
	\end{equation}
	where to evaluate $|s_i|$, we implicitly use a local trivializing section of $L$ coming from the choice of a model line bundle. In general, the NA Fubini-Study metrics can be defined as long as $s_1,\ldots, s_N$ generate $L^{\otimes l}\to X_K$; it is not necessary for these to form a basis of $V$.

	A continuous \emph{semipositive metric} $\norm{\cdot}=\norm{\cdot}_{\mathcal{L}}e^{-\phi}$ on $L$ can be then defined as a uniform limit of some sequence of NA Fubini-Study metrics. The space of continuous semi-positive potentials is denoted as $\text{CPSH}(X_K^{an}, \mathcal{L})$. We record a useful lemma for later convenience:
	
	\begin{lem}\label{lem:Lip}
		(Uniform Lipschitz estimate)
		Let $l\geq 1$, and let $L\to X$ be a relatively ample line bundle. Given a model $\mathcal{L}\to \mathcal{X}$, the valuation $x\mapsto val_x(s)$ for any nonzero $s\in H^0(X_K^{an}, lL)$ makes sense as a function on the dual complex $\Delta_{\mathcal{X}}$ using the local trivialisation of $\mathcal{L}$. On each face of $\Delta_{\mathcal{X}}$, we have a uniform estimate
		\[
		l^{-1}|val_x(s)- val_{x'}(s) | \leq C|x-x'|, \quad \forall x,x'\in \Delta_J,
		\]
		where the constant is independent of $x,x',s,l$.
	\end{lem}

	\begin{proof}
		The valuation function is piecewise affine linear on $\Delta_{\mathcal{X}}$, hence Lipschitz continuous, and the main question is the uniform estimate. We can find auxiliary sections $s_1,\ldots s_N$ in $H^0(X_K^{an}, lkL)$ for some large power $k$, so that $s_1,\ldots s_N$ have no common zero locus. Then 
		\[
		\phi= (kl)^{-1} \max( \log |s|^k , \max_i( \log |s_i|- c_i) ) 
		\]
		is a Fubini-Study potential, hence lies in $CPSH(X_K^{an}, L)$, hence by the equicontinuity result of Boucksom-Favre-Jonsson \cite[Thm. 6.1]{Boucksomsemipositive}, on any face of $\Delta_{\mathcal{X}}$, we have a uniform Lipschitz estimate $|\phi(x)- \phi(x') | \leq C|x-x'|$. However, by choosing $c_i\in \R$ large enough, then $\phi(x)= l^{-1}\log |s|(x)=- l^{-1}val_x(s)$ holds on $\Delta_{\mathcal{X}}$, hence the result. 
	\end{proof}

	To CPSH potentials we can associate the NA Monge-Amp\`ere measure $MA(\norm{\cdot})$, with total integral $(L^n)$. The central result of NA pluripotential theory is the solution of the NA Calabi conjecture:
	
	\begin{thm}\label{thm:BFJ}
		\cite{Boucksomsurvey}\cite{Boucksom} Let $L\to X$ be an ample line bundle. There is a unique (up to an additive constant) continuous semipositive metric $\norm{\cdot}_{CY,0}$ on $L$ such that $MA(  \norm{\cdot}_{CY,0} )=(L^n)\mu_0$, where $\mu_0$ is the normalised Lebesgue measure on $Sk(X)\subset X_K^{an}$. 
	\end{thm}

	We call $\norm{\cdot}_{CY,0}$ the NA Calabi-Yau metric. It can be characterised as the \emph{unique minimum of the convex  functional} $F_{\mu_0}: \text{CPSH}(X_K^{an}, \mathcal{L})\to \R$,
	\begin{equation}\label{eqn:Fmu01}
		F_{\mu_0}(\phi)= -  \frac{1}{(L^n)} E(\phi)+ \int_{X_K^{an}} \phi d\mu_0 =- \frac{1}{(L^n)}E(\phi)+ \int_{Sk(X)} \phi d\mu_0 .
	\end{equation}
	Here $E(\phi)$ is the \emph{Monge-Amp\`ere energy}, whose first variation is the NA MA measure.

	\subsection{Relative volume interpretation}\label{sect:relativevolume}

	Boucksom-Eriksson \cite{Boucksomnew1} gave an alternative formula for $E(\phi)$ in terms of relative volume as follows. Given two NA norms $\norm{}, \norm{}'$ on the $K$-vector space $V_l= H^0(X_K, lL)  $, the relative volume is
	\[
	vol(  \norm{}, \norm{}'  )=  \log \frac{  \det(\norm{}')  } {   \det(\norm{})   }.
	\]
	Concretely, if both $\norm{}$ and $\norm{}'$ are diagonal with respect to some NA orthogonal basis $s_1,\ldots s_N$ over the $K$-vector space $V$, namely
	\[
	\norm{ \sum a_i s_i}= \max_i |a_i| \norm{s_i}, \quad \norm{ \sum a_i s_i}'= \max_i |a_i| \norm{s_i}',\quad \forall a_i\in K,
	\]
	and $\norm{s_i}'= \lambda_i \norm{s_i}$, then $vol( \norm{}, \norm{}'    )= \log  \prod_i \lambda_i$.

	Given a continuous psh potentials $\phi, \psi \in \text{CPSH}(X_K^{an}, \mathcal{L})$, we can define a sequence of NA norms on the $K$-vector spaces $H^0(X_K, lL)$ for $l$ large enough,
	\[
	\norm{s}_{l\phi}= \sup_{X_K^{an}} |s|_{\mathcal{L}} e^{-l\phi},
	\]
	and similarly we define $\norm{s}_{l\psi}$. 
	By \cite[Theorem A]{Boucksomnew1}, the \emph{relative MA energy} admits the formula
	\[
	E(\phi)-E(\psi)= E(\phi, \psi)=  \lim_{l\to +\infty} \frac{n!}{ l^{n+1}} vol(   \norm{\cdot}_{l\phi}, \norm{\cdot}_{l\psi}     ) .
	\] 
	This determines $E(\phi)$ up to the choice of an additive constant.

	\subsection{NA MA measure vs. complex MA}\label{NAMAcxMA}

	The following `hybrid topology continuity principle' for the NA MA measure is essentially contained in the work of Favre \cite{Favre}, and further elucidated by Pille-Schneider \cite{PilleSchneider2}.

	We fix some ambient projective embedding of the degeneration family $X$ with ample polarisation $L\to X$, and fix a smooth Hermitian metric $h$ on  $L$ with positive curvature over the ambient space. The choice of $h$ does not matter. Then the Fubini-Study metric (\ref{NAFubiniStudy}) can be viewed as a hybrid topology limit as follows. 
	Given the data of the global sections $(s_1,\ldots s_N)$ in $H^0(X,lL)$ with no common zeros, and the real constants $c_1^l(t),\ldots c_N^l(t)$, we can then define the  continuous potentials on $(X_t,L)$,
	\[
	\phi_{l,t}= \frac{1}{l} \max_{1\leq i\leq N} (\log |s_i|_{h^{\otimes l}}-c_i^l(t) |\log |t||),
	\]
	corresponding to the continuous Hermitian metrics $h_{l,t}= he^{-2\phi_{l,t}}$ on $(X_t, L)$, which has non-negative curvature current on $X_t$. For later application, we comment that $s_1,\ldots s_N$ \emph{does not} need to be a basis of $H^0(X,lL)$. Suppose $t=t_j$ is a sequence tending to zero, and $c_i^l(t)\to c_i^l$ as $j\to +\infty$ for fixed $i,l$. We recall the NA Fubini-Study metric in (\ref{NAFubiniStudy}). Then for each fixed $l$,
	\begin{itemize}
		\item The potential $\frac{1}{|\log |t||}\phi_{l,t}\to \phi_{FS,l}$ as $t=t_j\to 0$ in the $C^0$-hybrid topology. Equivalently, the archimedean norm functions $|s|_{h_{l,t}}^{1/|\log |t||}\to \norm{s}_{FS,l}$ in $C^0$-hybrid topology for any local algebraic sections $s$ of $L$. 
		
		\item  The complex MA measure on $X_t$ associated to $h_{l,t}$ converges weakly as $t=t_j\to 0$ to the NA MA measure on $X_K^{an}$. (This is based on interpreting the concentration of complex MA measures in terms of intersection numbers, see \cite[Prop. 3.1]{Favre}.)
	\end{itemize}

	The more general potentials can be handled by uniform Fubini-Study approximation. Let $\phi_t$ be continuous $\omega_h$-psh potentials on $(X_t, L)$.  We assume the following uniform approximation property:  there exists small $r_l>0$, and a double sequence of approximants $\phi_{l,t}$ as above with $l\to +\infty$ and $t=t_j\to 0$, such that 
	\begin{equation}\label{uniformapproximation1}
		\sup_{X_t} | \phi_{l,t}-\phi_t| \leq \epsilon_l |\log |t||,\quad \forall  0<|t|\leq r_l,
	\end{equation}
	and $\epsilon_l\to 0$ as $l\to +\infty$. Then there is a continuous NA psh potential $\phi$ on $(X_K^{an}, L)$ which arises as the $C^0$-limit of $\frac{1}{|\log |t||}\phi_t$ as $t=t_j\to 0$. Moreover, the complex MA measures on $X_t$ associated to $\phi_t$ converge to the NA MA measure of $\phi$ as $t=t_j\to 0$. (This is based on the Chern-Levine inequality, see \cite[section 4]{Favre}.)

	\begin{rmk}
		Favre \cite[section 4]{Favre} required the small radii $r_l>0$ to be independent of $l$. This is not necessary if we a priori assume $\phi_t$ to have the $\omega_h$-psh property. Notice we do not need to require $\phi_t$ to have psh dependence on the family parameter $t$. The intention of treating $t$ as a sequence rather than a continuous variable, is that we shall later use some diagonal argument to extract some subsequence such that $c_i^l(t)\to c_i^l$. 
	\end{rmk}

	%\begin{rmk}
	%The potential $\phi_{l,t}$ above deviates from the usual Fubini-Study potential
	%\[
	%\frac{1}{2l} \log \sum_{i=1}^N |s_i|_{h^{\otimes l}}^2 |t|%^{-2c_i^l(t)}
	%\]
	%by an amount bounded by $\frac{\log N(l)}{l}$, which converges to zero as $l\to +\infty$. So from the viewpoint of $C^0$ convergence of potential, both are interchangeable. 
	%\end{rmk}

	\section{Valuative independence and Canonical basis}\label{sect:canonicalbasis}

	\subsection{Valuative independence condition}\label{sect:valuativeindependence}

	We shall now digest the valuative independence condition in Def. \ref{Def:valuativeindependence}. Let $l\geq 1$, and let $s_1,
	\ldots s_N$ be holomorphic sections of $L^{\otimes l}\to X$, which are meromorphic over the field of convergent Laurent series in $t$. Upon base change, we can regard $s_i$ as sections over $X_K$ defined over the formal punctured disc.  By the GAGA principle, we can regard $s_i$ as sections of $L\to X_K^{an}$. Upon choosing some local trivialising section $\tau$ of $L\to X_K$, they induce local functions $s_i/\tau$, so we can take their NA norm with respect to the monomial valuation defined by any given point $x\in Sk(X)\subset X_K^{an}$.

	\begin{Def}\label{def:valuativeindependenceatx} 
		Given $x\in Sk(X)$, we say valuative independence holds for $s_1,\ldots s_N$ at $x$, if for any coefficient function $a_i\in K$, 
		\begin{equation}\label{eqn:valuativeindependence2}
			| \sum_i a_i s_i/\tau | (x)=  \max_i  |a_i|  |s_i/\tau|(x).
		\end{equation}
		Here $|a_i|= e^{- val(a_i)}$ uses the standard discrete valuation on $K$, which measures the vanishing order of $a_i$, and is independent of $x$. Notice that (\ref{eqn:valuativeindependence2}) does not depend on the choice of the local trivialising section $\tau$, since the norm factor of $\tau$ cancels out from both sides of the equation; we will take license to omit $\tau$ in the formula.
		
	\end{Def}

	Valuative independence can be viewed as a strong form of linear independence.

	\begin{lem}\label{lem:linearindependence}
		Suppose $s_i$ satisfy Def. \ref{def:valuativeindependenceatx}, then they are linearly independent over the field $K$. 
	\end{lem}

	\begin{proof}
		Suppose $\sum_i a_i s_i=0$ with $a_i\in K$. Taking the valuation at any $x\in Sk(X)$, 
		\[
		0= | \sum_i a_i s_i| (x)=  \max_i  |a_i|  |s_i|(x),
		\]
		so $a_i=0\in K$. 
	\end{proof}

	In our setting, $X_t$ is Calabi-Yau, and $L\to X_t$ is ample, by Kodaira vanishing theorem and Riemann-Roch, the dimension
	\[
	N(l)= \dim H^0(X_t, lL)= \dim H^0(X_t, lL+K_{X_t})= \chi(X_t, lL+K_{X_t})= \chi(X_t, lL)
	\]
	is constant in $t$.  Let $\theta_1,\ldots \theta_{N(l)}$ be holomorphic sections of $L^{\otimes l}\to X$, which are meromorphic over the field of convergent Laurent series in $t$. Tautologically,

	\begin{lem}
		The sections $\theta_1,\ldots \theta_{N(l)}$ satisfy Def. \ref{Def:valuativeindependence} if and only if valuative independence holds  at every $x\in Sk(X)$.
	\end{lem}

	\begin{cor}
		For small $t\in \mathbb{D}^*$, the restriction of $\theta_i$ form a basis of $H^0(X_t, lL)$. 
	\end{cor}

	\begin{proof}
		The linear independence of $\theta_i |_{X_t}\in H^0(X_t, lL)$ can fail only on a closed analytic subvariety, so it holds after possibly shrinking the disc. By dimension counting, $\theta_i|_{X_t}$ form a basis.
	\end{proof}

	\subsection{Equivalent characterisation}\label{sect:equivalentcharacterisation}

	Given the sections $s_1,\ldots s_N$, we wish to concretely determine when they satisfy the valuative independence condition at $x\in Sk(X)$. We denote $m=\dim Sk(X)$, and throughout the paper we assume $1\leq m\leq n$.  Let $(\mathcal{X}, \mathcal{L})\to S$ be an SNC model of $(X,L)$, and let $E_J=\cap_{i\in J} E_i$ be a depth $m+1$ intersection of SNC divisor components $E_i$ on the central fibre, such that the divisorial valuations for $E_i$ achieve the minimum value of the log discrepancy, hence the corresponding simplex $\Delta_J$ in the dual complex lies on an $m$-dimensional face of the essential skeleton $Sk(X)$. 
	Let $z_i$ be the local equations of the divisors $E_i$, and we can arrange $t=z_0^{b_0}\ldots z_m^{b_m}$.

	Upon fixing a choice of local trivialising section $\tau$ for the line bundle, we can regard the sections of $L$ as local meromorphic functions which are holomorphic for small $t\neq 0$, and Taylor expand
	\[
	s/\tau=\sum_{ \alpha\in \Z^{m+1} } f_\alpha z_0^{\alpha_0} \ldots z_m^{\alpha_m},
	\]
	where $f_\alpha$ are functions of the local coordinates of $E_J$, and negative $\alpha_i$ terms appear at most finitely many times. The monomial valuation of $s$ at the interior point $x\in \Delta_J\subset Sk(X)$ is
	\[
	-\log 
	|s|(x)=val_x(s) = \min \{    \langle x,\alpha\rangle   |f_\alpha\neq 0  \}. 
	\]

	\begin{cor}
		On any given face $\Delta_J$, the function $\log |s|$ is the maximum of a finite collection of integral affine linear functions, and in particular is piecewise affine linear.
	\end{cor}
	
	For any fixed $l$, and any nonzero section $s$ of $L^{\otimes l}\to X_K^{an}$, 
	the valuation functions $\log |s|$ have uniform Lipschitz constants independent of $s$ by Lemma \ref{lem:Lip}, so modulo adding an integer constant, there are only a finite number of possibilities for these integral affine linear functions. In particular, the corner locus of $\log |s|$ is contained in a finite union of codimension one walls $\mathcal{S}\subset Sk(X)$. We denote the connected components of $Sk(X)\setminus \mathcal{S}$ as $O_k$, which are polyhedral open domains in some $m$-dimensional face $\Delta_J^0$. On each $O_k$, the valuation functions $\log |s_i|$ are affine linear in $x$, and correspondingly only one monomial term $f_{i,\alpha} z_0^{\alpha_{0,i}}\ldots z_m^{\alpha_{m,i}}$ dominates all the other terms in the Taylor expansion of $s_i$. Given a different choice of the local defining equations $z_i$ for $E_i$, then the Taylor expansion would be different, but the exponent of the leading monomial term $(\alpha_{0,i}, \ldots, \alpha_{m,i})$ would be unchanged.  We partition the exponents into equivalence classes modulo integer multiples of $(b_0,\ldots b_m)$.

	%One can check that this condition does not depend on the choice of local defining functions $z_i$, nor does it depend on the local chart of $E_J$, or the choice of the local trivialising section $\tau$. 

	\begin{lem}(Equivalent characterisation)
		\label{lem:valuativeindependence2}
		Given any connected component $O_k$, 
		Def. \ref{def:valuativeindependenceatx} holds for $s_1,\ldots s_N$ at every $x\in O_k$, if and only if the following holds: whenever the leading exponents $(\alpha_{0,i},\ldots ,\alpha_{m,i})$ for some subcollection of $s_i$ fall within the same equivalence class modulo $\Z(b_0,\ldots b_m)$, the coefficient functions $f_{i,\alpha}$ are $\C$-linearly independent as local functions on $E_J$.

	\end{lem}

	\begin{proof}
		We first show the if direction.
		%Since the valuations of any section $s$ defines a continuous function on $Sk(X)$, it suffices to prove valuative independence only for $x\in \Delta_J^0\setminus \mathcal{S}$. 
		Given any coefficient functions $a_i\in K$ for $i=1,\ldots N$, we need to show
		\[
		val_x (\sum a_i s_i)= \min_i (val(a_i) +  val_x(s_i))=  \min_i (val(a_i) +  \sum_k \alpha_{k,i} x_k).
		\]
		Let $1\leq i\leq N$ be an index achieving the minimum in the RHS, then it suffices to show that the term 
		\[
		t^{val(a_i)} z_0^{\alpha_{0,i}}\ldots  z_m^{\alpha_{m,i}} = z_0^{\alpha_{0,i}+ b_0 val(a_i)}\ldots  z_m^{\alpha_{m,i}+b_m val(a_i)}
		\]
		appears with nonzero coefficient function in the Taylor expansion of $\sum a_j s_j$. By construction, the subleading monomials have higher valuation on $O_k$, so can be neglected, and the exponents in other equivalence classes mod $\Z(b_0,\ldots b_m)$ cannot contribute to this term. Among the indices in the same equivalence class, the above linear independence condition for $f_{i,\alpha}$ ensures that after taking linear combination, the coefficient function cannot be identically zero.

		For the converse direction, suppose the $\C$-linear relation $\sum_i c_i f_{i,\alpha}=0$ holds within some equivalence class of $s_i$. We can find some integers $val(a_i)$, such that the monomial
		\[
		t^{val(a_i)} z_0^{\alpha_{0,i}}\ldots  z_m^{\alpha_{m,i}} = z_0^{\alpha_{0,i}+ b_0 val(a_i)}\ldots  z_m^{\alpha_{m,i}+b_m val(a_i)}
		\]
		are the same for these $s_i$. Then taking $a_i= c_i t^{val(a_i)}$, in the expansion of $\sum a_is_i$ the coefficient function for the above monomial is zero, so $val_x(\sum a_is_i)$ is strictly bigger than $\min_i (val(a_i) +  val_x(s_i))$, violating the valuative independence condition. 
	\end{proof}

	\begin{prop}\label{prop:localvaluativeindependence}
		(Local valuative independence)
		Given any connected component $O_k\subset Sk(X)\setminus \mathcal{S}$, we can find a $K$-basis of sections $s_1^{(k)}, \ldots s_{N(l)}^{(k)}$, such that Def. \ref{def:valuativeindependenceatx} is satisfied for any $x\in O_k$.
	\end{prop}

	\begin{proof}
		We inductively enlarge the collection $s_1,\ldots s_N$ while keeping the valuative independence condition at $x$. %By Lemma \ref{lem:linearindependence}, valuative independence implies $K$-linear independence. 
		Suppose the contrary that a collection $s_1,\ldots, s_N$ with $N<N(l)=\dim H^0(X_t, lL)$ cannot be enlarged.

		For any section $s$ of $L^{\otimes l}\to X_K^{an}$ not in the $K$-span of $s_1,\ldots s_N$, we can extract its leading Taylor expansion term $f_\alpha z_0^{\alpha_0}\ldots z_m^{\alpha_m}$ over  $O_k$.  By the characterisation of Lemma \ref{lem:valuativeindependence2}, and the hypothesis that $s_1,\ldots s_N, s$ fail to satisfy Def. \ref{def:valuativeindependenceatx}, we infer that 
		\[
		f_\alpha z_0^{\alpha_0}\ldots z_m^{\alpha_m}=\sum c_it^{n_i} f_{i,\alpha} z_0^{\alpha_{0,i}}\ldots z_m^{\alpha_{m.i}}
		\] for a suitable equivalence class of exponents $(\alpha_{0,i},\ldots \alpha_{m,i}) $, and $c_i\in \C$, $n_i\in \Z$. Consequently,  this leading Taylor expansion term cancels out in $s-\sum c_i t^{n_i} s_i$, hence $val_x(s-
		\sum c_i t^{n_i} s_i)> val_x(s)$ at any $x\in O_k$. Repeating this process, we deduce that $s=\sum a_i s_i$ for suitable $a_i\in K$, contradiction.
	\end{proof}

	\begin{rmk}\label{rmk:choiceofbasiss_i^k}
		Any other choices of such $K$-basis $\tilde{s}_1^{(k)}, \ldots , \tilde{s}_{N(l)}^{(k)}$ is related to $s_1^{(k)}, \ldots s_N^{(k)}$ as follows:
		\begin{itemize}
			\item  We classify the sections $\tilde{s}_i^{(k)}, s_i^{(k)}$ according to the leading order exponent $(\alpha_0,\ldots \alpha_m)$ modulo $(b_0,\ldots b_m)$, on the connected component $O_k$. Within each equivalence class, up to multiplying by powers of $t$, we may assume $(\alpha_0,\ldots, \alpha_m)$ are the same.

			\item  We write $\tilde{s}_j^{(k)} = \sum_i s_i^{(k)} B_{ij}$, for some transformation matrix $B\in GL(N(l), K)$. Applying the valuative independence condition at any point $x\in O_k$, we deduce for any $1\leq i,j\leq N(l)$ that
			\[
			val_x(\tilde{s}_j^{(k)} )\leq val_x( s_i^{(k)}  )+ val(B_{ij}).   %,\quad  	val_x(s_j^{(k)} )\leq val_x( \tilde{s}_i^{(k)}  )+ val((B^{-1})_{ij}).
			\]
			
			Moreover, if equality is achieved at one interior point $x_0\in O_k$, then by the affine linearity of $val_x(s)$ as a function of $x\in O_k$, the equality must be achieved for all $x\in O_k$. This means the two valuation functions differ by an integer, or equivalently the leading order exponents of $s_j^{(k)}$ and $\tilde{s}_i^{(k)}$ fall inside the same equivalence class. By our preliminary normalisation, these leading order exponents are equal, hence $val(B_{ij})=0$. 
			The same kind of conclusions hold for $B^{-1}$. 
			
			\item Now we construct the `leading order matrix' $\bar{B}$ (and likewise $\overline{B^{-1}}$), which are size $N(l)$ square matrices with entries in $\C$. For each $i,j$, if the strict inequality holds for $x\in O_k$,
			\[
			val_x(\tilde{s}_j^{(k)} )< val_x( s_i^{(k)}  )+ val(B_{ij}),
			\]
			then we set $\bar{B}_{ij}=0$. Hence $\bar{B}_{ij}$ vanishes when the leading exponents of $\tilde{s}_j^{(k)} $ and $s_i^{(k)} $ fall in different equivalence classes, namely the matrix $\bar{B}_{ij}$ is `block diagonal'.
			Otherwise we are in the equality case, and we set $\bar{B}_{ij}$ to be the constant term of the power series $B_{ij}(t)$.

			\item We start with the matrix identities 
			\[
			\sum_a B_{ia} B^{-1}_{aj} = \delta_{ij}, 
			\]
			and work modulo $O(t)$. For each pair $i,j$ such that $s_i^{(k)}$ and $s_j^{(k)}$ lie in the same equivalence class of exponents, then
			\[
			\begin{split}
				& val(B_{ia}B^{-1}_{aj})= val(B_{ia}) + val(B^{-1}_{aj}) 
				\\
				& \geq  -val_x(s_i^{(k)}) + val_x( \tilde{s}_a^{(k)} ) -val_x( \tilde{s}_a^{(k)} )+ val_x(s_j^{(k)}) =0,
			\end{split}
			\]
			and equality only holds when $\tilde{s}_a^{(k)}$ is also in the same equivalence  class. Thus the leading order term gives $\bar{B} \overline{B^{-1}}=I$, so $\bar{B}\in GL(N(l), \C)$.

			\item Applying the transformation matrix $\bar{B}$ to the $K$-basis $s_i^{(k)}$, then the new basis will have the same leading order term $f_\alpha z_0^{\alpha_0}\ldots z_m^{\alpha_m}$ as $\tilde{s}_i^{(k)}$, for $i=1,\ldots N(l)$. 
			The remaining transformation matrix $\bar{B}^{-1}B$ fixes the leading order terms, and can be viewed as a perturbation of the identity matrix.

			\item  Suppose there exists some basis $\theta_1,\ldots \theta_{N(l)}$ satisfying the valuative independence condition Def. \ref{Def:valuativeindependence}, one may ask \emph{to what extent is $\theta_1,\ldots \theta_{N(l)}$ a canonical basis}. Now Def. \ref{Def:valuativeindependence} is equivalent to saying that Def. \ref{def:valuativeindependenceatx} holds for all $x\in O_k$ simultaneously for all the choices of $O_k$. Thus up to multiplying $\theta_i$ by some powers of $t$, we have the freedom to apply a `block diagonal' transformation matrix in $GL(N(l),\C)$, respecting the equivalence classes of exponents on all connected components $O_k$. The remaining freedom in the transformation matrix $B$ is a `perturbation of the identity', namely 
			\[
			B=I+B',\quad  val(B'_{ij}) > val_x(\theta_j)- val_x(\theta_i) ,\quad \forall x\in \cup_k O_k. 
			\]
			
		\end{itemize}

	\end{rmk}

	\subsection{Filtration by valuative constraints}\label{sect:filtration}

	We now use \emph{valuative constraints} to define a multi-graded filtration on the $K$-vector space $H^0(X_K^{an},lL)$. Given any connected component $O_k\subset Sk(X)\setminus \mathcal{S}$, Prop. \ref{prop:localvaluativeindependence} provides the basis sections $s_1^{(k)},\ldots, s_N^{(k)}$ satisfying Def. \ref{def:valuativeindependenceatx} for any $x\in O_k$.   For any $s\in H^0(X_K^{an}, lL)$, we can expand $s=\sum_{i=1}^{N(l)} a_i^k s_i^{(k)}$ with $a_i^k\in K$. Let $\underline{\mu}=(\mu_{i,k})$ be an integer vector, then we can define the filtered piece 
	\[
	F_{\underline{\mu} } = \{     s\in  H^0(X_K^{an}, lL) : val(a_i^k) \geq \mu_{i,k} ,\forall 1\leq i\leq N(l), \forall k     \},
	\]
	which is a module over the ring $R= \C[\![t]\!]$. The union of all these filtered pieces is $H^0(X_K^{an}, lL)$.
	Given two integer vectors $\underline{\mu}, \underline{\mu}'$, we say $\underline{\mu}'\geq \underline{\mu}$ (resp. $\underline{\mu}'> \underline{\mu}$), if $\underline{\mu}'-\underline{\mu}$ have non-negative integer entries (resp. excluding equality). Clearly $F_{\underline{\mu}'}  \subset F_{\underline{\mu}} $ if $\underline{\mu}'\geq \underline{\mu}$. Let the multigraded pieces be the $\C$-vector space
	\[
	gr_{\underline{\mu} }  = F_{\underline{\mu}} /\oplus_{ \underline{\mu}'> \underline{\mu} }  F_{ \underline{\mu}' }.
	\]
	Any nontrivial element in $gr_{\underline{\mu}}$ represented by some section $s$, satisfies $val(a_i^k) =\mu_{i,k}$. Thus $gr_{\underline{\mu}}$ is at most one-dimensional.

	The graded pieces satisfy a periodicity condition due to the action of $t$ on $H^0(X_K^{an}, lL)$, which induces an $R$-module isomorphism $F_{\underline{\mu}} \simeq t F_{\underline{\mu}}= F_{ \underline{\mu}+ \underline{1}  }$, where $\underline{1}= (1,\ldots 1)$. Each periodicity equivalence class has a unique representative $\underline{\mu}$ satisfying $\min_{i,k} \mu_{i,k}=0$.

	\begin{lem}
		If $gr_{\underline{\mu}}$ is nonzero and $\min_{i,k} \mu_{i,k}=0$, then  $\max_{i,k } \mu_{i,k}\leq C$.
	\end{lem}

	\begin{proof}
		Given an SNC model
		$\mathcal{X}\to S$, and some section $s$ representing a nonzero element of the graded piece,
		then $val_x(s)$ can be regarded as a function on $Sk(X)$, and its restriction to any face of $Sk(X)$ has a uniform Lipschitz estimate independent of $s\in H^0(X_K^{an}, lL)$ by Lemma \ref{lem:Lip}. The condition that $\min_{i,k} val(a_i^k)= \min_{i,k} \mu_{i,k}=0$ then implies that 
		\[
		\mu_{i,k} = val(a_i^k) \leq C,\quad \forall i,k,
		\]
		as required.
	\end{proof}

	\begin{cor}
		The $\C$-vector space $\bigoplus_{ \underline{\mu}:  \min_{i,k} \mu_{i,k}=0  } gr_{\underline{\mu}}$ is finite dimensional.
	\end{cor}

	\begin{proof}
		There are only a finite number of choices of integer vectors $\underline{\mu}$ satisfying $\min_{i,k} \mu_{i,k}=0$ such that the graded piece is nonzero, and each summand $gr_{\underline{\mu} }$ is at most one dimensional. 
	\end{proof}

	We take a $\C$-basis of each summand $gr_{\underline{\mu}}$ ranging over all choices of $\underline{\mu}$ with $\min_{i,k} \mu_{i,k}=0$. These can be lifted to a collection of meromorphic sections $S_1,\ldots S_N$ of $H^0(X, lL)$ with some finite radius of convergence.

	\begin{lem}\label{lem:decompositionfiltration}
		Any section $s\in F_\mu$ can be written as a linear combination
		\[
		s= \sum_1^N  f_i(t) S_i, \quad f_i(t)\in K,
		\] 
		such that all the summands $ f_i(t) S_i \in F_{\underline{\mu}}$.
		In particular, the sections $S_1,\ldots S_N$ generate $H^0(X_K^{an}, lL)$ as a $K$-vector space, so $N\geq N(l)=\dim H^0(X_t, lL)$.

		%	More precisely, for any $s\in F_{\underline{\mu}
		\end{lem}

		\begin{proof}
			Without loss $gr_{\underline{\mu}}$ is nonzero, for otherwise $s$ can be decomposed into a finite linear combination with summands belonging to a strictly higher filtered piece. But by the periodicity action, there is some $\underline{\mu}'= \underline{\mu}+ a\underline{1}$, with $\min_{i,k} \mu'_{i,k}=0$, and $a\in \Z$. We can subtract $t^{-a}$ times a suitable multiple of the basis section  $S_i\in gr_{\underline{\mu}'}$, to improve $s$ to a higher filtered piece. Within at most $N$ steps, we can reduce $s$ to the higher filtration $tF_{\underline{\mu}}$. The result then follows from induction on degree.
		\end{proof}

		\begin{rmk}
			If $s$ is a meromorphic section, then we can assume $f_i(t)$ is meromorphic for small nonzero values of $t$. 
		\end{rmk}

		\begin{rmk}
			Suppose  there exists sections $\theta_1,\ldots \theta_{N(l)}$ satisfying Def. \ref{Def:valuativeindependence}. Then we can choose $s_i^{(k)}$ to be $\theta_i$, for all choices of connected components $O_k$ of $Sk(X)\setminus \mathcal{S}$. Then  $\theta_1,\ldots \theta_{N(l)}$ gives a basis of the $\C$-vector space $\bigoplus_{ \underline{\mu}:  \min_{i,k} \mu_{i,k}=0  } gr_{\underline{\mu}}$. In particular $N=N(l)=\dim H^0(X_t, lL)$. The moral is that $S_1,\ldots S_N$ is a weaker substitute for such `canonical basis'  $\theta_1,\ldots \theta_{N(l)}$, with the caveat that $N$ may be in general bigger than $N(l)$.
		\end{rmk}

		\subsection{Sections with valuative constraints}\label{sect:valuativeconstraints}

		We now relate the non-archimedean notion of valuative independence, to the algebraic geometry of sections $s\in H^0(X_t, l L)$, for $t$ sufficiently small. Given any connected component $O_k$, Prop. \ref{prop:localvaluativeindependence} provides the basis sections $s_1^{(k)},\ldots, s_{N(l)}^{(k)}$, which can be assumed to converge for small finite values of $t$, and define a basis of $H^0(X_t, lL)$.  We can expand $s$ as a linear combination of the sections in $H^0(X_t, lL)$,
		\[
		s= \sum_{i=1}^{N(l)} a_i^k s_i^{(k)}, \quad a_i^k\in \C.
		\]
		A natural way to measure the magnitude of $s$ is to impose bounds on $|a_i^k|$. 
		We will use big $O$ notation to denote some estimable constants which are independent of small enough $t$. 
		We are interested in the system of valuative constraints
		\begin{equation}\label{eqn:linearinequality}
			|a_i^k|=O( |t|^{\mu_{i,k}} ), \quad \forall i, k,
		\end{equation}
		where $\mu_{i,k}\in \R$ are prescribed exponents. %This can be viewed as a system of linear inequalities on the coefficients of $s$. 

		We begin with a general lemma on such linear inequalities.

		\begin{lem}\label{lem:linearinequality1}
			Let $b_{ij}$ be prescribed meromorphic series of $t$, and $\mu_{i}\in \R$ be prescribed exponents. For small values of $t$, consider the system of linear inequalities
			\[
			\sum_{j=1}^J b_{ij} x_j = O( |t|^{\mu_{i}} ), \quad i=1,\ldots I.
			\]
			There exists an invertible matrix in $B\in GL(J, \C[\![t]\!])$ whose entries are represented by power series of $t$ with finite radius of convergence, such that the system is equivalent to the new system for $x'= Bx$:
			\[
			x_i'= O(|t|^{\mu_i'}), \quad i=1,\ldots I',\quad I'\leq \min(I, J).
			\]
			Here the set of exponents $\{ \mu_1',\ldots \mu_{I'}'  \}\subset \{ \mu_1,\ldots \mu_I \}+\Z$. 
		\end{lem}

		\begin{proof}
			(Sketch) The strategy is to simplify the system by row transformations using matrices in $GL(I,\C)$, and multiplying rows by  powers of $t$.

			By dividing each row with a suitable power of $t$, without loss $b_{ij}$ are all power series of $t$ (namely there is no negative power term), and each leading order row vector  $b_{ij}(0)$ is a nonzero vector. We order the rows in the decreasing order of $\mu_{i}$.

			Similar to Gaussian elimination, we can subtract $\C$-linear combinations of previous rows from latter rows; the point is that $O(|t|^\mu)$ can be absorbed into $O(|t|^{\mu'})$ if $\mu\geq \mu'$. Repeating this process, we will reach the situation that the first $I''\leq I$ row vectors $b_{ij}(0)$ are linearly independent, and all the other rows have $b_{ij}(0)=0$, namely $b_{ij}$ is divisible by $t$ as a power series. We then divide out $t$, reorder the rows in decreasing order of $\mu_i$, and continue the process; the exponents $\mu_i$ will decrease by positive integers.

			Eventually, the system takes the following form: the first $I'\leq I$ rows of the matrix $(b_{ij}(0))\in Mat_{I\times J}(\C)$ are linearly independent, while the last $I-I'$ row vectors $b_{ij}$ are linear combinations of the first $I'$ rows, over the power series ring. We can then subtract the linear combinations, to make $b_{ij}=0$ for the last $I-I'$ rows. The first $I'$ rows can be completed into the required invertible matrix $B$.
		\end{proof}

		\begin{rmk}
			The construction of the matrix $B$ depends on $\mu_i$ only through the ordering of the rows. Even though $\mu_i$ are continuous parameters, different choices of $\mu_i$ only result in a finite number of different choices for $B$. 
		\end{rmk}

		\begin{cor}\label{cor:linearinequality2}
			Let $b_{ij}$ be prescribed meromorphic series of $t$, and $\mu_{i}\in \R$ be prescribed exponents. For small values of $t$, consider the system of linear inequalities
			\[
			\begin{cases}
				\sum_{j=1}^J b_{ij} x_j \leq C |t|^{\mu_{i}} , \quad i=1,\ldots I,
				\\
				\sum c_j   x_j= 1.
			\end{cases}
			\]
			There is some small $t_0>0$ depending on $b_{ij}$ but not on $c_j, \mu_i$, such that the following holds. 
			If there is some $0< |t'|\leq t_0$, such that the system admits some solution, 
			then there exist meromorphic functions $f_j(t)$ and some $\nu\in \R$, such that $x_j(t)= |t|^\nu f_j(t)$ satisfies the following:
			\[
			\begin{cases}
				|	\sum_{j=1}^J b_{ij}(t) |t|^\nu f_j (t)  |\leq C |t|^{\mu_{i}} , \quad i=1,\ldots I, \quad \forall 0<|t|\leq t_0. 
				\\
				\sum_j c_j  |t'|^\nu f_j(t')= 1.
			\end{cases}
			\]
			Here the constant $C\geq 1$ depends on the matrix $(b_{ij})$, but does not depend on $t,\mu_i, c_j$.

		\end{cor}

		\begin{proof}
			In terms of the new variables $x'= Bx$ provided by Lemma \ref{lem:linearinequality1}, the system reads 
			\[
			\begin{cases}
				x_i' = O( |t|^{\mu_{i}'}) , \quad i=1,\ldots I',
				\\
				\sum_{j,k} c_j B_{jk}(t')  x_k' (t') = 1.
			\end{cases}
			\]
			If there is some $k> J-I'$, such that $\sum c_j B_{jk}(t')\neq 0$, then we can set the other entries of $x'$ to zero, and take the ansatz $x_k'(t)= c |t|^\nu $, and define $\nu$ by
			\[
			|t'|^{-\nu}= | \sum c_j B_{jk}(t')|.
			\] 
			Then $x'(t)$ solves the system for some choice of $c\in U(1)$.
			Thus we may assume $J=I'$ so there is no `unconstrained variable'.

			The existence of solution at $t=t'$ implies that 
			\[
			\max_k (  |t'|^{\mu_k'}  	|\sum_{j} c_j B_{jk}(t') | )   \geq C^{-1}\max_k |\sum_j c_j B_{jk}(t')   x_k'| \geq   C^{-1}  | \sum_{j,k} c_j B_{jk}(t') x_k' |=  C_0^{-1}.
			\]
			The maximum is achieved for some $k$. We now set all the other enties of $x'$ to zero, and take the ansatz $x_{k}'(t)= c C_0 |t|^\nu$, and define $\nu$ by
			\[
			|t'|^{-\nu}=C_0 | \sum c_j B_{jk}(t')|.
			\] 
			Thus $|t'|^{\mu_k'- \nu} \geq 1$, whence $\mu_k'\leq \nu$. Then $x'(t)$ solves the system for some $c\in U(1)$. The solution translates back to $x_j(t)= B^{-1}_{jk}(t) x_k'(t)= cC_0|t|^\nu B^{-1}_{jk}(t)$, so $f_j(t)= cC_0 B^{-1}_{jk}(t)$ is up to a bounded constant factor equal to a column of $B^{-1}(t)$. 
		\end{proof}

		We now apply the general theory of linear inequalities to the system of valuative constraints (\ref{eqn:linearinequality}) on sections. 	Recall that in Section \ref{sect:filtration} we constructed a finite collection of meromorphic sections $S_1,\ldots S_N$ in $H^0(X,lL)$.

		\begin{cor}(`From fibre to family')\label{cor:fromfibretofamily}
			There is a small enough $t_0>0$ such that the following holds. Suppose for some $t'$ with $0<|t'|\leq t_0$, there is some section $s'\in H^0(X_{t'}, lL)$ satisfying the valuative constraints 
			\[
			s'= \sum a_i^{k'} s_i^{(k)},\quad 	|a_i^{k'}|\leq C|t'|^{\mu_{i,k}} , \quad \forall i, k.
			\]
			Let $z\in X_{t'}$ be an arbitrary prescribed point.

			Then there is some meromorphic section $\tilde{s}$ of the line bundle $L^{\otimes l}\to X$ holomorphic over a small punctured disc, and some $\nu \in \R$, such that for all $t$ with $0<|t|\leq t_0$, the restriction of $ \tilde{s}$ to $X_{t}$ satisfies the valuative constraints 
			\[
			\tilde{s}=  \sum a_i^k s_i^{(k)}     ,\quad 	|a_i^k|(t)\leq C|t|^{\mu_{i,k}-\nu} , \quad \forall i, k,
			\]
			and at the prescribed point $z\in X_{t'}$ we have $|t'|^\nu\tilde{s}(z)=s'(z)$.

			Furthermore, we can require $\tilde{s}= cS_\alpha$ for some $ |c|\leq C$, and some $S_\alpha$ from the collection. The constants $C$ are uniform for small $t$, the choice of $\mu_{i,k}$, and the point $z\in X_{t'}$.
		\end{cor}

		\begin{proof}
			For different choices of $k,k'$, the transformation matrix between the two $K$-basis $s_1^{(k)},\ldots, s_{N(l)}^{(k)}$ and $s_1^{(k')},\ldots, s_{N(l)}^{(k')}$ is a matrix whose entries are meromorphic series of $t$. Thus for $k\neq k_0$, the $a_i^k$ are linear expressions of $a_i^{k_0}$. The constraint (\ref{eqn:linearinequality}) can be thus viewed as a system of linear inequalities on the unknown $a_i^{k_0}$, of the general form $\sum b_{ij}x_j= O(|t|^{\mu_i})$.

			The evaluation $s(z)$ can be viewed as a linear functional on $s$, so prescribing a nonzero value $s(z)$ amounts to solving an equation $\sum c_i a_i^{(k_0)}=1$. By 
			Cor. \ref{cor:linearinequality2}, we can find some meromorphic section $\tilde{s}_{old}$ and $\tilde{\nu}\in \R$, satisfying the valuative constraints
			\[
				\tilde{s}_{old}=  \sum \tilde{a}_i^k s_i^{(k)}     ,\quad 	|\tilde{a}_i^k|(t)\leq C|t|^{\mu_{i,k}-\tilde{\nu}} , \quad \forall i, k, \quad \forall 0<|t|\leq t_0,
			\]
		and the prescribed evaluation $|t'|^{\tilde{\nu}} \tilde{s}_{old}(z)= s'(z)$.

			The last part of the proof of Cor. \ref{cor:linearinequality2} gives a little more information about $\tilde{s}_{old}$. Up to a bounded constant factor, it agrees with a column vector for some fixed matrix $B^{-1}(t)$. The matrix $B(t)$ is determined by the fixed choice of the sections $s_i^{(k)}$ for all the connected components $O_k\subset Sk(X)\setminus \mathcal{S}$, together with a choice of ordering of the rows (coming from the decreasing order condition for the exponents $\mu_i$). The upshot is that up to the bounded constant factor, $\tilde{s}_{old}$ can be required to belong to a \emph{finite set} of meromorphic sections.

			By the valuative constraints on $\tilde{s}_{old}$, 	as an element of $H^0(X_K^{an}, lL)$, it belongs to the filtration $F_{\underline{\mu'}}$, where the integer grading
			$
			\mu_{i,k}'= \lceil   \mu_{i,k}- \tilde{\nu}  \rceil.
			$
			Applying Lemma \ref{lem:decompositionfiltration}, we can write
			\[
			\tilde{s}_{old}= \sum_1^N f_\alpha(t) S_\alpha,
			\]
			such that all the summands lie in the filtration $F_{\underline{\mu'}}$, and $f_\alpha(t)= t^{n_\alpha} h_\alpha(t)$ are meromorphic for small $t\neq 0$, such that $n_\alpha\in \Z$, and $h_\alpha(t)$ is a power series with nonzero constant term, unless it is identically zero.

			At the point $z\in X_{t'}$, the triangle inequality gives
			\[
			|\tilde{s}_{old}(z)| \leq C\max_{1\leq \alpha\leq N} |f_\alpha(t') S_\alpha(z)| \leq C \max_{h_\alpha\neq 0} |t'|^{n_\alpha} |S_\alpha(z)|.
			\]
			The modulus here is measured with respect to any fixed choice of local trivialisation of $L$; the choice does not matter since both sides cancel out. Since $\tilde{s}_{old}$ belongs to a finite set up to bounded constant, the constant $C$ is uniform. We take some $1\leq \alpha\leq N$ achieving this maximum.

			We take the ansatz
			\[
			\tilde{s}= cS_\alpha, \quad \nu= \tilde{\nu}+ n_\alpha,\quad  c\in \C, \quad	|c| |t'|^{n_\alpha} |S_\alpha(z)| =  |\tilde{s}_{old}(z)|.
			\]
			By the above discussion $|c|\leq C$. Thus
			\[
		|t'|^\nu |\tilde{s}(z)|=  |t'|^\nu |c| |S_\alpha(z)|= |t'|^{\tilde{\nu}} |\tilde{s}_{old}(z)| =  |s'(z)|, 
			\]
			hence by choosing the phase factor of $c$ suitably, we achieve the point evaluation $	|t'|^\nu \tilde{s}(z)= s'(z)$. Moreover, 
		since $t^{n_\alpha }S_\alpha$ belongs to the filtration $F_{\underline{\mu}'}$, the valuative constraint on $\tilde{s}$ is satisfied.  
		\end{proof}

		\subsection{Examples}

		The existence of $\theta_1,\ldots \theta_{N(l)}$ satisfying the \emph{global} valuative independence condition Def. \ref{Def:valuativeindependence}, is only known in a few examples.

		The intuition of valuative independence is that $\theta_i$ behaves like a \emph{monomial basis} up to higher order correction. This can be explained in the local example of an affinoid torus fibration.

		\begin{eg}
			(Affinoid torus fibration) We start with the Berkovich space attached to the family $(\C^*)^{n+1}_{z_0,\ldots z_n}\to \C^*_t$ given by $t=z_0\ldots z_n$. There is a tropicalisation map to $\{ \sum x_i=1   \}\subset \R^{n+1}$ defined by the coordinates
			\[
			x_i= -\log |z_i|, \quad \sum_0^n x_i= -\log |t|=1.
			\]
			We consider the open subset of the Berkovich space, defined as the preimage of the open simplex $\Delta^0= \{   x_i>0, \forall i\}\subset\{  \sum_0^n x_j=1    \}$.  Via the Gauss embedding, $\Delta^0$ defines an open subset of the essential skeleton.

			The monomial basis $z^\beta= z_1^{\beta_1}\ldots z_n^{\beta_n}$ for $\beta\in \Z^n$ is a basis of the $\C(t)$-vector space $\C[z_i^{\pm 1}, i=0,\ldots n]$. For any meromorphic  $f=\sum_{\alpha\in \Z^{n+1}  } f_\alpha z^\alpha$, we can write $f=\sum_{\beta\in \Z^n} a_\beta z^\beta$ with $a_\beta\in K$, and for any $x\in \Delta^0$, we can compute the monomial valuation
			\[
			val_x( f  )= \min \{   \sum x_j \alpha_j |  f_\alpha\neq 0  \}= \min_\beta (val(a_\beta)+ val_x(z^\beta ) ).
			\]
			This means a local version of the valuative independence condition is satisfied. If we modify $z^\beta$ by adding higher order perturbations, the valuations will not be affected, so the valuative independence is preserved.

		\end{eg}

		We now move on to some examples for compact CY manifolds.

		\begin{eg}\label{eg:intermediate}
			(Intermediate complex structure limit \cite{Liintermediate}) Let $M$ be an $(n+1)$-dimensional smooth Fano manifold, and we write its anticanonical bundle as $-K_M=(\sum_0^m d_i)L$
			for some ample line bundle $L\to M$, and natural numbers $d_i>0$. Let $F_0, F_1,\ldots, F_m$ (resp. $F$) be sufficiently generic sections of $H^0(M, L^{\otimes d_i})$ (resp. $H^0(M, -K_M)$), so that all the divisors $E_i=(F_i=0)$ and $(F=0)$ are smooth, and all intersections are transverse. We consider the family of CY hypersurfaces $X\to \C_t^*$:
			\begin{equation}\label{Fanohypersurface}
				X_t= \{  F_0F_1\ldots F_m +tF=0   \} \subset M.
			\end{equation}
			This family extends over $t=0$ and defines a dlt minimal model $\mathcal{X}\to \C_t$. 
			Under the assumption $1\leq m\leq n-1$, the intersection $E_J=\cap_0^m E_i$ is connected using the Lefschetz hyperplane theorem. Moreover all the divisor  components $E_i$ have multiplicity $b_i=1$, so  the essential skeleton $Sk(X)$ is the standard $m$-dimensional simplex,
			\begin{equation}
				Sk(X)\simeq \Delta= \{  \sum_0^m x_i=1, x_i\geq 0     \}\subset \R^{n+1}.
			\end{equation}

			We now construct a basis $\theta_i^l$ using the ideas in \cite[section 3.1]{Liintermediate}. Repeatedly using exact sequence and Kodaira vanishing arguments, we see the restriction
			\[
			H^0(M, kL)\to H^0(E_J, kL), \quad k\geq 0,
			\]
			is surjective, so we can pick a subspace $V_k\subset H^0(M,kL)$ mapping isomorphically to $H^0(E_J, kL)$, and we select $\tau_a^k$ to be a basis of $V_k$. For any sequence of integers $(l_0,\ldots l_m)$ satisfying
			\begin{equation}\label{eqn:liexponents}
				l_i\geq 0,\quad \sum_0^m d_i l_i \leq l, \quad \min_{i} l_i=0,
			\end{equation}
			we take the sections
			\[
			F_0^{l_0}\ldots F_m^{l_m} \otimes \tau^{l-\sum d_il_i}_a \in H^0(M, lL) ,
			\]
			and regard these as holomorphic sections on $X$ by restriction. We define $\{ \theta_i^l\}$ to be the union of all these sections when we range  over all such sequences $(l_0,\ldots l_m)$.

			\begin{lem}
				For each $l\geq 0$, the number of such sections is $N(l)= \dim H^0(X_t, lL)$. 
			\end{lem}

			\begin{proof}
				The strategy is to compare the generating series of these numbers. For each $l$, the number of all such sections is
				\[
				\sum_{(l_0,\ldots l_m)}  \dim V_{l-\sum d_i l_i}=   \sum_{(l_0,\ldots l_m)}  \dim H^0(E_J, (l-\sum d_il_i)L) ,
				\]
				so the generating series is 
				\[
				\sum_l \sum_{(l_0,\ldots l_m)}  \dim H^0(E_J, (l-\sum d_il_i)L) y^l.
				\]
				The condition (\ref{eqn:liexponents}) requires $\min_i l_i=0$. The partial generating series obtained by summing all terms with $l_0=0$ is
				\[
				\begin{split}
					&	(1+ y^{d_1}+ y^{2d_1}+\ldots )\ldots (1+ y^{d_m}+y^{2d_m}+ \ldots )(\sum_k   \dim H^0(E_J, kL)y^k )
					\\
					=	& 
					(1- y^{d_1})^{-1} \ldots (1- y^{d_m})^{-1} P_{E_J}(y)
					\\
					=&  (1- y^{d_0}) P_M(y)	,
				\end{split}
				\]
				where $P_{E_J}(y)$ and $P_M(y)$ are the Hilbert-Poincar\'e series of $(E_J,L)$ and $(M,L)$ respectively, and we are applying the formula for the Hilbert series of complete intersections. Similarly, for any given $0\leq i_1<\ldots <i_s\leq m$, the partial sum over all terms with $l_{i_1}=\ldots l_{i_s}=0$ is $(1- y^{d_{i_1}})\ldots (1- y^{d_{i_s}}) P_M(y)$. By inclusion-exclusion, the generating series is
				\[
				\begin{split}
					&	\sum_{s\geq 1} \sum_{0\leq i_1<\ldots <i_s\leq m}  (-1)^{s-1} (1- y^{d_{i_1}})\ldots (1- y^{d_{i_s}}) P_M(y)
					\\
					=&   (1- y^{\sum d_i}) P_M(y)
					\\
					=&  P_{X_t} (y)= \sum_k \dim H^0(X_t, kL) y^k.
				\end{split}
				\] 
				The third line uses that $X_t$ is a degree $\sum d_i$ hypersurface in $M$. 
				The Lemma follows by comparing coefficients.
			\end{proof}

			To check that this collection of sections satisfies valuative independence, we appeal to Lemma \ref{lem:valuativeindependence2} (which works for dlt models just as in the  SNC case, since the dual complex of the dlt model is just the dual complex of its SNC locus). The collection of exponents are the $(l_0,\ldots l_m)$ specifed by (\ref{eqn:liexponents}); this collection has the property that each equivalence class modulo $\Z(1,\ldots 1)$ contains a unique exponent $(l_0,\ldots l_m)$, and for this given exponent, the sections $\tau^{l-\sum d_il_i}_a$ are linearly independent as sections on $E_J$.

		\end{eg}

		\begin{eg}\label{eg:abelian}
			(Abelian variety) We now follow \cite[Section 2]{Gross} to consider Mumford's construction of theta functions on Abelian varieties. The starting data is a lattice $M\simeq \Z^n$, $M_\R = M \otimes_\Z \R$, $N = Hom_\Z(M,\Z)$, a sublattice $\Gamma\subset M$, %a $\Gamma$-periodic polyhedral decomposition $\mathcal{P}$ of $M_\R$, 
			and a strictly convex piecewise linear function with integral slopes $\Phi: M_\R\to\R$ satisfying a periodicity condition, for $\gamma\in \Gamma$,
			\[
			\Phi(m+\gamma)= \Phi(m)+\alpha_\gamma(m),
			\]
			for some affine linear function $\alpha_\gamma$ depending on $\gamma$. From this data one builds an unbounded polyhedron in $M_\R \oplus \R$: 
			\[
			\Delta_\Phi := \{(m,r)|m \in M_\R,r \geq \Phi(m)\}.
			\]
			The normal fan of this polyhedron in $N_\R\oplus  \R$ is a fan $\Sigma_\Phi$ with an infinite number of cones, defining a toric variety $X_\Phi$ which is not of finite type.
			Further, $\Gamma$ acts on $N \oplus \Z$; indeed, $\gamma\in \Gamma$ acts by
			taking $(n, r) \mapsto (n-r  d\alpha_\gamma , r)$, where $d\alpha_\gamma$ denotes the differential of $\alpha_\gamma$. This action preserves $\Sigma_\Phi$.

			The projection $N_\R\oplus \R \to \R$ defines a map $\pi:X_\Phi \to \C$. The fibres of this map are algebraic tori $(\C^*)^n$ except for $\pi^{-1}(0)$, which is an infinite union of proper toric varieties. Furthermore, the action of $\Gamma$ preserves this map, and yields an action of $\Gamma$ on the irreducible components of $\pi^{-1}(0)$.
			While the $\Gamma$-action is global, it does not act properly discontinuously except on the subset $\pi^{-1}(D)$, where $D \subset \C$ is the unit disk. Thus we get a family $\pi: \mathcal{X}= \pi^{-1}(D)/\Gamma\to D$, 
			whose general fibre is an abelian variety and such that the fibre over zero is a union of toric varieties.

			The polyhedron $\Delta_\Phi$ induces a line bundle $\mathcal{L}\to X_\Sigma$, which descends to a line bundle on the quotient family $\mathcal{L}\to \mathcal{X}$.  For any level $l\geq 1$, to write down sections of $\mathcal{L}^{\otimes l}\to \mathcal{X}$, it suffices to write down $\Gamma$-invariant sections of $\mathcal{L}^{\otimes l}\to X_\Sigma$ subject to the convergence requirements. This can be done by taking, for each $m\in l^{-1}M$, the infinite series
			\[
			\vartheta_m= \sum_{\gamma\in \Gamma} z^{   (  l(m+\gamma), l\Phi(m+\gamma))       }.
			\]
			Here $( l(m+\gamma), l \Phi(m+\gamma) ) $ is an integral point in $(M\oplus \Z)\cap l \Delta_\Phi$, so corresponds to a monomial section of $\mathcal{L}^{\otimes l}\to X_\Sigma$. One observes that $\vartheta_{m+\gamma}=\vartheta_m$, so if we set $B= M_\R/\Gamma$, we get a set of theta functions indexed by the points of $B(l^{-1}\Z)$. On the abelian variety fibres, these $\vartheta_m$ restrict to the basis of classical theta functions.

			We now check that for each $l\geq 1$, the valuative independence condition holds for  the theta basis $\{ \vartheta_m \}$. For any pair $m\neq m'\in B(l^{-1}\Z )$, the monomial exponents $( l(m+\gamma), l \Phi(m+\gamma) ) $ and $( l(m'+\gamma'), l \Phi(m'+\gamma') ) $ lie in distinct equivalence classes modulo $\Z(0,1)$ (namely the monomial corresponding to $t$), so Lemma \ref{lem:valuativeindependence2} implies valuative independence.

		\end{eg}

		\subsection{Theta functions in mirror symmetry*}\label{sect:theta}

		This section will be of purely motivational character.

		\begin{Question}
			Given a polarised degeneration family, what is the source of the distinguished basis of sections for $L^{\otimes l}\to X$?
		\end{Question}

		In the Abelian variety example \ref{eg:abelian}, we saw that the classical theta functions provide a canonical basis satisfying the valuative independence condition. 
		A vast generalisation of the theta function construction has been found in the context of the Gross-Siebert programme of mirror symmetry, in various degrees of generality \cite{Grosstheta1}\cite{Grosstheta2}\cite{Grosstheta3}. The broad picture is the following \footnote{As a caveat, our notations for $X$ and $X^\vee$ will be the opposite to the conventions of the Gross-Siebert programme, since for us the theta functions will live on $X$ rather than $X^\vee$.}:
		
		\begin{itemize}
			\item One starts with a polarised family of Calabi-Yau manifolds $(X^\vee, L^\vee)\to S$ near the large complex structure limit,  and constructs the essential skeleton $B$, which is an $n$-dimensional polyhedral complex with a natural integral structure.     (In a variant setting, one can start with a log Calabi-Yau pair $(X^\vee,D)$).

			\item  The mirror to $X^\vee\to S$ will be a polarised family of projective varieties $(X,L)$ over a formal neighbourhood in a multi-parameter base $H^2(X^\vee)$. When $(X^\vee, L^\vee)\to S$ 
			is a toric degeneration of Calabi-Yau varieties \cite[Def. 7.1]{Gross}  with polarisation data, then $(X,L)$ is again a toric degeneration of Calabi-Yau varieties with polarisation data, whose central fibre is obtained by gluing toric varieties along toric strata, and the family $X$ is obained by successively modifying the gluing map between toric charts order by order, subject to the Konsevich-Soibelman wall crossing formula. In many cases, the polarised family $(X,L)$ over the formal neighbourhood extends to a meromorphic family with finite radius of convergence.

			\item The theta functions $\{  \vartheta_p^l \}$ are a \emph{canonical basis} of sections for $L^{\otimes l}\to X$, for $l\geq 1$. For each $l$, the theta functions are indexed by the rational points  $p\in B(l^{-1}\Z)$. In the toric degeneration setting, the theta functions can be combinatorially constructed using counts of `broken lines' in $B$.

			\item The multiplication rules for theta functions can be defined either via punctured Gromov-Witten invariants, or equivalently using combinatorial counts of tropical trees with two inputs and one output. For $p\in B(l^{-1}\Z), p'\in B(l'^{-1}\Z)$, the product is
			\[
			\vartheta_p^l \vartheta_{p'}^{l'}=  \sum_{r\in B(  (l+l')^{-1}\Z  )} \alpha_{p,p',r} \vartheta_r^{l+l'},
			\]
			where the coefficient is a formal sum
			\[
			\alpha_{p,p',r}= \sum_{A\in H_2(X^\vee)} N_{p,p',r}^A t^A.
			\]
			Here $N_{p,p',r}^A$ are enumerative invariants, which are only nonzero when $A$ lives inside the monoid of effective curve classes, and in particular the energy of the curve is non-negative. The Gross-Siebert mirror family $(X,L)$ agrees with the relative Proj construction of the ring of theta functions. (In the variant setting of log Calabi-Yau pairs $(X^\vee, D)$, one instead takes the Spec of the ring of theta functions over the formal multi-parameter base.)

			\item The central fibre of Gross-Siebert mirror family is obtained by keeping only the zero energy $A=0$ contribution. By \cite[Lemma 1.15, Prop. 1.16]{Grosstheta3}, the coefficient $N_{p,p',r}^{A=0}$ is only non-zero if $p,p'$ live on the same face of $B$, and $r= \frac{l}{l+l'} p+ \frac{l'}{l+l'} p'$, in which case $N_{p,p',r}^{A=0}=1$.

		\end{itemize}

		A large source of examples of large complex structure limits of Calabi-Yau manifolds comes from the base change of the multi-parameter family in the Gross-Siebert programme. 
		In the light of the above, it would be interesting to ask:
		
		\begin{Question}
			When do the theta functions in the Gross-Siebert programme satisfy valuative independence condition? When this condition holds, can one explicitly compute the valuation of the theta functions on the essential skeleton $Sk(X)$?
		\end{Question}

		\begin{rmk}
			The theta functions are defined over the formal multi-parameter base, due to the usual convergence issue of curve counting invariants when one applies Gromov compactness. On the other hand, valuative independence is only sensitive to the leading order terms in the Taylor expansion of the theta functions, so for our purpose one can take truncations of the theta functions instead. 
		\end{rmk}

		\begin{rmk}
			This question is mostly unknown for polarised degeneration of compact Calabi-Yau manifolds, but in the related context of cluster varieties, which are special cases of the log Calabi-Yau setting above, a version of the valuative independence for theta functions has been observed in the exciting recent work of Cheung-Magee-Mandel-Muller, announced in \cite{Mandel}. 
		\end{rmk}

		Gross-Siebert programme so far has been mainly focused on the large complex structure limit. Example \ref{eg:intermediate} suggests there is more to the story.

		\begin{Question}
			For polarised degenerations of Calabi-Yau manifolds which are \emph{not} large complex structure limits, when is there a canonical basis of sections for $L^{\otimes l}\to X$, and when does it satisfy valuative independence?
		\end{Question}

		\section{Potential convergence in the hybrid topology}\label{sect:potentialconvergence}

		The goal of this section is to prove Theorem \ref{thm:hybridconvergence}. The key tool is to compare the Fubini-Study approximation with the Calabi-Yau metric on $(X_t,L)$, and the proof will share a number of key ingredients with the intermediate complex structure limit example \cite{Liintermediate}.

		Let $(\mathcal{X},\mathcal{L})\to S$ be a semistable SNC model of $(X,L)$. Let $h$ be a fixed positively curved Hermitian metric on $(\mathcal{X},\mathcal{L})$ over some disc around $0\in \mathbb{D}_t$. The curvature current of the Hermitian metric $h_{CY,t}=h e^{-2\phi_{CY,t}}$ defines the Calabi-Yau metric on $(X_t,L)$.

		\begin{lem}\label{lem:Linfty}
			\cite{LiuniformSkoda} Under the normalisation $\sup_{X_t} \phi_{CY,t}=0$, we have the $L^\infty$ bound $\norm{\phi_{CY,t}}_{C^0(X_t)}\leq C|\log |t||$, where the constant is independent of $t$. 
		\end{lem}

		\subsection{Bergman kernel approximation}

		We start with an application of Ohsawa-Takegoshi extension, following the key ingredient in \cite[Section 6.3]{Liintermediate}.

		\begin{prop}\label{OhsawaTakegoshi}
			For any sufficiently large $l$, and
			for any $z\in X_t$, there is a section $s\in H^0(X_t, lL)$ with 
			\[
			\frac{1}{2l}\log \int_{X_t} |s|^2_{h^{\otimes l }} e^{ -2l \phi_{CY,t} } \sqrt{-1}^{n^2}\Omega_t \wedge \overline{\Omega}_t \leq    ( \frac{1}{l} \log |s|_{h^{\otimes l}} -\phi_{CY,t}  )(z) + \frac{C}{l} |\log |t||.
			\]	
			The constant $C$ is uniform for all small $t$ and large $l$.
		\end{prop}

		\begin{proof}
			The idea is to apply the Ohsawa-Takegoshi extension theorem \cite[Section 12.C]{Demailly} to produce some section of $lL+ K_{X_t}$ with norm control, by extending from the given point $z\in X_t$ to $X_t$. We select the positive metric on the line bundle $lL$ as a weighted combination of the CY metric $h e^{ -2\phi_{CY,t} } $ and the background metric $h$,
			\[
			h^{\otimes N_0} (h e^{ -2\phi_{CY.t} } ) ^{\otimes (l-N_0) }= h^{\otimes l}e^{ -2(l-N_0) \phi_{CY,t} },
			\]
			where $N_0$ is a large fixed integer independent of $l$, and $\tau_0,\ldots \tau_{N_1}\in H^0(M, N_0 L)$ induces a projective embedding of $\mathcal{X}$, such that $h$ is comparable to the Fubini-Study metric induced by the projective embedding. Since both the CY potential and the background metric $h$ are K\"ahler, the $h^{\otimes N_0}$ term gives a \emph{fixed positive lower bound} on the curvature independent of $l$.

			Any $z\in X_t$ can be viewed as a point in the ambient space $\mathbb{P}^{N_1}$, and without loss it lies in the affine coordinate chart with $\tau_0=1, |\tau_i |\leq 1$, for all $i=1,\ldots N_1$. Up to linear change of coordinates by a matrix bounded in $GL(N_1+1,\C)$, we can assume $\tau_1(z)=\ldots= \tau_{N_1}(z)=0$. Moreover, we can select $\tau_1,\ldots \tau_n$ so that in the local affine coordinates, the function
			\[
			\frac{ d\tau_1\wedge \ldots \wedge d\tau_n}{ \Omega_t }= \frac{ d\tau_1\wedge \ldots \wedge d\tau_n\wedge dt}{ \Omega }
			\]
			is meromorphic on $\mathcal{X}$ and holomorphic for small nonzero $t$, so its modulus is bounded below by $|t|^{\kappa}$ for  some fixed power $\kappa\in \R$. Since the metric $h$ is uniformly equivalent to one in the local affine coordinates, we get
			\[
			|\frac{ d\tau_1\wedge \ldots \wedge d\tau_n }{  \Omega_t} |_{h^{\otimes nN_0}}(z) \geq  C^{-1} e^{-\kappa |\log |t||},
			\]
			and for $\tilde{\tau}=C'^{-1}(\tau_1,\ldots \tau_n)\in H^0(X_t, \oplus_1^n L^{\otimes N_0})$, we can arrange
			$
			|\tilde{\tau} |_{h^{\otimes N_0  }} \leq \frac{1}{2}.
			$
			The point is that the constants are independent of $z$.

			The Ohsawa-Takegoshi extension theorem \cite[Thm. 12.6]{Demailly} applied to $\{  \tau_1=\ldots =\tau_n=0 \}\subset X_t$ then produces a section $s\in H^0(X_t, lL)$ (or equivalently a section $s\otimes \Omega_t\in H^0(X_t, lL+K_{X_t})$), which is nonvanishing at $z$ and vanishes at all other points of $\{  \tau_1=\ldots =\tau_n=0 \}\subset X_t$, satisfying the quantitative estimate
			\[
			\begin{split}
				& \int_{X_t} |s|^2_{h^{\otimes l}} e^{ -2(l-N_0) \phi_{CY,t} } \sqrt{-1}^{n^2}\Omega_t \wedge \overline{\Omega}_t \\
				& \leq 
				C\int_{X_t}  \frac{ |s|^2_{h^{\otimes l}} }{   |\tilde{\tau}|_{h^{\otimes N_0  }}^{2n} (-\log |\tilde{\tau}|_{h^{\otimes N_0  }})^2    }e^{ -2(l-N_0) \phi_{CY,t}  } \sqrt{-1}^{n^2}\Omega_t \wedge \overline{\Omega}_t
				\\
				& \leq C  |s|^2_{h^{\otimes l}}(z) e^{- 2(l-N_0) \phi_{CY,t}(z)  } | \frac{    d\tau_1\wedge \ldots \wedge d\tau_n   }{  \Omega_t} |^{-2}_{h^{ \otimes nN_0 }}(z)
				\\
				& \leq  |s|^2_{h^{\otimes l}}(z) e^{- 2(l-N_0) \phi_{CY,t}(z)  } e^{C|\log |t||}.
			\end{split}
			\]
			Now $|\phi_{CY,t}|\leq C|\log |t||$ by Lemma \ref{lem:Linfty}, 
			hence 
			\[
			\int_{X_t} |s|^2_{h^{\otimes l}} e^{ -2l \phi_{CY,t} } \sqrt{-1}^{n^2}\Omega_t \wedge \overline{\Omega}_t \leq  (|s|^2_{h^{\otimes l}} e^{- 2l \phi_{CY,t} }) (z) e^{C|\log |t||}
			\]
			and taking logarithm proves the claim.
		\end{proof}

		\subsection{Average potentials}\label{sect:averagepotential}

		We will now make some auxiliary constructions to analyse the K\"ahler potential in the generic region. The intuition is that modulo a small measure subset, the local oscillation of $\phi_{CY,t}$ is small. Let $\delta>0$ be a small number to be prescribed later.

		By the CY volume computation of section \ref{section:essentialskeleton}, after deleting a closed subset with normalised CY measure $
		\ll 
		\delta$, the rest of $X_t$ is covered by the local charts around the depth $m+1$ intersections strata $
		E_J=\cap_0^m E_i$ for $E_i$ achieving the minimum value of log discrepancy, but staying $O(\delta)$-bounded away from deeper strata, as in section \ref{section:essentialskeleton}. We take $z_0,\ldots z_m$ as the local defining equations for $E_i$, where $z_0\ldots z_m=t$, and let $x_i= -\frac{\log 
			|z_i|}{|\log |t||}$, so $x\in \Delta_J$. Let $\tilde{U}_{\delta}$ be the union over all choices of $E_J$ for the preimage of the subset $\{   x_i\gtrsim \delta, 1-\sum_0^m x_i\gtrsim \delta  \}$, then the normalised CY measure of its complement is bounded by $C \delta$.

		We let $z_{m+1},
		\ldots z_n$ denote the local coordinates on a polydisc $D(\delta)^{n-m}$ in $E_J$ which stays $O(\delta)$ bounded away from the deeper strata, so $z_1,
		\ldots z_n$ gives a local coordinate system on $X_t$. The number of such polydisc charts needed to cover the complement of deeper strata in $E_J$ is bounded by some constant depending on $\delta$, but not on $t$. 
		The CY measure $\sqrt{-1}^{n^2}\Omega\wedge \overline{\Omega}$ is $C(\delta)$-uniformly equivalent to 
		\[
		\bigwedge_1^m \sqrt{-1} d\log z_i\wedge d\overline{\log z_i}\wedge \bigwedge_{m+1}^n \sqrt{-1} dz_j\wedge d\bar{z}_j.
		\]

		Let $\phi$ be any potential such that $\omega_h + dd^c \phi$ is K\"ahler, and $\norm{\phi}_{L^\infty}\leq C|\log |t||$. 
		In such a coordinate system, we can take the \emph{average potential}
		\begin{equation}\label{eqn:Lipphibar}
			\begin{split}
				&	\bar{\phi}(x_1,\ldots x_m):=
				\\
				& \dashint_{D(\delta)^{n-m} \times  T^n}  \phi(  e^{-|\log |t||x_1+ i\theta_1} , \ldots   e^{-|\log |t||x_m+i\theta_m} , z_{m+1},\ldots   )\bigwedge _1^m d\theta_j \wedge  \bigwedge_{m+1}^n \sqrt{-1} dz_j\wedge d\bar{z}_j.
			\end{split}
		\end{equation}
		This defines a function for $(x_1,\ldots x_m)$ inside $\Delta_J$. We notice that the local potential $\phi_h$ of the Hermitian metric $h$ has $C^\infty$ bounds in the $z_1,\ldots z_n$ coordinates. If we replace $\phi$ by the local psh potential $\phi_h+\phi$, then the average function would be a \emph{convex} function as in \cite[Lemma 2.2]{LiuniformSkoda}, so the $L^\infty$ bound implies a Lipschitz bound on the shrunken simplex $\{   x_i\gtrsim \delta, 1-\sum_0^m x_i\gtrsim \delta  \}\subset \Delta_J$. The Lipschitz bound then transfers to $\phi$ as well,
		\begin{equation}
			|\bar{\phi}(x)- \bar{\phi}(x')|\leq C(\delta) |\log |t|| |x-x'|.
		\end{equation}

		We wish to compare $\phi$ and $\bar{\phi}$ in some integral sense.

		\begin{lem}
			($L^2$-gradient estimate)  Let 
			\[
			\omega_{Eucl}=  \frac{1}{|\log |t||}  \sqrt{-1}  \sum_1^m  d\log z_j\wedge d\overline{\log z_j} + \sum_{m+1}^n  \sqrt{-1} dz_j\wedge d\bar{z}_j ,
			\]
			then
			\[
			\int     |\nabla \phi|_{\omega_{Eucl}}^2 \omega_{Eucl}^n 		    \leq C_1(\delta) |\log |t|| .
			\]
			The integral here is over the subset $ D(\delta)^{n-m}\times \{  x_i\gtrsim \delta, 1-\sum_0^m x_i\gtrsim \delta     \}$ inside $\tilde{U}_\delta$, and $C_1(\delta)$ is some constant depending on $\delta$. 
		\end{lem}

		\begin{proof}
			By \cite[Prop. 3.1]{Lidiameter} we can find some global K\"ahler metric $\omega_t'$ on $X_t$ in the class $ c_1(L)$, such that
			\[
			\omega_t'\geq C(\delta)^{-1}   \frac{1}{|\log |t||} \sum_1^m  d\log z_j\wedge d\overline{\log z_j} 
			\]
			on the region of interest. On the other hand, a suitable Fubini-Study metric in the class $c_1(L)$ is bounded below by $C^{-1}\sum dz_i\wedge \overline{dz_i}$, so by linear interpolation, we can require 
			\[
			\omega_t'\geq C(\delta)^{-1}   \frac{1}{|\log |t||} \sum_1^m  d\log z_j\wedge d\overline{\log z_j}  + C^{-1}  \sum dz_i\wedge \overline{dz_i} \geq C(\delta)^{-1}\omega_{Eucl}.
			\]
			Using integration by parts,
			\[
			\int_{X_t} d\phi\wedge d^c\phi \wedge \omega_t'^{n-1}  \leq C(\delta)\norm{\phi}_{L^\infty} \int_{X_t} (\omega_h+ dd^c \phi) \wedge  \omega_t'^{n-1} \leq C(\delta)\log |t||.
			\]
			Thus
			\[
			\int d\phi\wedge d^c \phi  \wedge \omega_{Eucl}^{n-1} \leq C(\delta) |\log |t|| ,
			\]
			which is equivalent to the claim. We comment that $\omega_{Eucl}^n$ is uniformly equivalent to the normalised CY measure on $\tilde{U}_\delta$. 
		\end{proof}

		We can cover the region $D(\delta)^{n-m}\times \{  x_i\gtrsim \delta, 1-\sum_0^m x_i\gtrsim \delta     \}$ by \emph{dyadic scales} of the form
		\[
		D(\delta)^{n-m}\times \{    R_i\leq |z_i|\leq 2R_i , i=1,\ldots m     \},
		\]
		for varying choices of dyadic parameters $(R_1,\ldots R_m)$. The total number of dyadic scales is bounded by $C |\log |t||^m$.    We say a dyadic scale is \emph{bad} if 
		\[
		\int_{\text{dyadic}} d\phi\wedge d^c \phi  \wedge \omega_{Eucl}^{n-1} \geq C_1(\delta)\delta^{-1} |\log |t||^{1-m}.
		\]
		By a counting argument, the number of bad dyadic scales is at most $\delta |\log |t||^m$, so the total contribution to the normalised CY measure is bounded by $C\delta$ when $t$ is small enough depending on $\delta$.

		On the good scales, we thus gained an $L^2$-gradient estimate of the form
		\[
		\int_{\text{dyadic}} |\nabla \phi|_{\omega_{Eucl}}^2 \omega_{Eucl}^n 		
		\leq C(\delta) |\log |t||^{1-m}.
		\]
		By the Poincar\'e inequality,
		\begin{equation}\label{eqn:phiphibarL2}
			\int_{dyadic}  |\phi- \bar{\phi}|^2 \omega_{Eucl}^n \leq  C(\delta)  |\log |t||^{1-m}.
		\end{equation}
		Equivalently, the average integral
		\[
		\dashint_{dyadic}  |\phi- \bar{\phi}|^2 \omega_{Eucl}^n \leq  C(\delta)  |\log |t||.
		\]
		
		\begin{cor}\label{cor:largedeviation}
			There is some constant $C_2(\delta)$ large enough depending on $\delta$, such that
			the normalised CY measure satisfies 
			\[
			\mu_t(	\{  \zeta\in \tilde{U}_\delta: | \phi- \bar{\phi} |\geq C_2(\delta) |\log |t||^{1/2}  \}  )\leq C\delta.
			\]

		\end{cor}

		\begin{proof}
			On each good dyadic scale, by the Chebyshev estimate and (\ref{eqn:phiphibarL2}),
			\[
			\mu_t(     	\{  \zeta:  |\phi- \bar{\phi} |\geq C_2(\delta) |\log |t||^{1/2}  \}           ) \leq  \frac{C(\delta)}{C_2(\delta)^2}  |\log |t||^{-m}.
			\]
			The total number of dyadic scales is bounded by $C(\delta)|\log |t||^m$, while the total normalised CY measure in the bad dyadic scales is bounded by $C\delta$. Hence by choosing $C_2(\delta)$ large enough, we can ensure the claimed measure estimate. 
		\end{proof}

		We note that within any dyadic scale the oscillation of $\bar{\phi}$ is bounded by $C(\delta)$ by (\ref{eqn:Lipphibar}); the intuition is that $\bar{\phi}$ behaves like constants in each dyadic scale. By the plurisubharmonicity $\omega_h+ dd^c\phi\geq 0$,  we deduce from (\ref{eqn:phiphibarL2}) that within any good scale, 
		\begin{equation}\label{eqn:subharmonic}
			\phi \leq  \bar{\phi} + C_3(\delta) |\log |t||^{1/2}
		\end{equation}
		holds on the \emph{shrunken dyadic scale} 
		\begin{equation}\label{eqn:shrunkendyadic}
			D(\delta/2)^{n-m}\times \{   2R_i/3\leq |z_i|\leq 3R_i/2   ,i=1,\ldots m    \}.
		\end{equation}

		\subsection{$L^2$-almost orthogonality}\label{sect:L2almostorthogonality}

		We continue with the notations of Section \ref{sect:averagepotential}. In the $x_i$-coordinates, we let $\mathcal{S}_
		\delta$ be the subset of $\tilde{U}_\delta$ whose $(x_i)$ coordinates lie in the 
		$\delta$-neighbourhood of the corner locus $\mathcal{S}$ in section \ref{sect:equivalentcharacterisation}. The total normalised CY measure in this region is bounded by $C\delta$.

		From henceforth, let $l\geq l_0$ be large enough, so that the space of sections generate $L^{\otimes l}\to X_t$ for all $t\in \mathbb{D}^*$. Let $O_k$ be a connected component of $Sk(X)\setminus \mathcal{S}$, contained in some $m$-dimensional face $\Delta_J^0$. By Prop. \ref{prop:localvaluativeindependence}, we can find a $K$-basis of sections 
		$s_i^{(k)}$ with $i=1,2,\ldots N(l)$, satisfying valuative independence for all $x\in O_k$. The following is a complex geometric manifestation of the valuative independence condition, which is based on the orthogonality of Fourier series. %It shows that linear independence of $\theta_i\in H^0(X_t, lL)$ holds in a strong way.
		
		\begin{lem}\label{lem:asymptoticorthogonality1}
			Given $l\geq l_0$, and small $\delta>0$, the following holds when $t$ is small enough depending on $l,\delta$. For any $a_1,\ldots a_{N(l)}\in \C$, and for any $\zeta$ in any given dyadic scale in $\tilde{U}_\delta$ disjoint from $\mathcal{S}_\delta$, whose $(x_i)$-coordinates lies in $O_k$, we have
			\[
			\int  |\sum_i a_i s_i^{(k)}|_{h^{\otimes l}}^2    
			\sqrt{-1}^{n^2}\Omega_t\wedge \overline{\Omega}_t \geq  C(\delta,l)^{-1}\sum_i  |a_i|^2 |s_i^{(k)}|_{h^{\otimes l}}^2  (\zeta)
			\]
			where the LHS integral is over the shrunken dyadic scale (\ref{eqn:shrunkendyadic}). 
		\end{lem}

		\begin{proof}
			The constants in this proof will depend on $\delta, l$. We omit the superscript $k$ for convenience.
			In the local trivialisation of $
			\mathcal{L}\to \mathcal{X}$, the potential of the Hermitian metric $h$ is bounded by some constant independent of $t$. Thus we shall regard $s_i$ as local holomorphic functions instead of sections, and we can replace the $h$-norm by the modulus. It suffices to show
			\begin{equation}
				\int |\sum_{i=1}^{N(l)} a_i s_i|^2 \bigwedge_1^m \sqrt{-1} d\log z_i\wedge d\overline{\log z_i}\wedge \bigwedge_{m+1}^n \sqrt{-1} dz_j\wedge d\bar{z}_j  \geq  C^{-1}\sum_i  |a_i|^2 |s_i|^2(\zeta). 
			\end{equation}
			The LHS  integration is over the shrunken dyadic scale, which is $T^m$-invariant.

			Since the dyadic scale is disjoint from $\mathcal{S}_\delta$, we have ensured that each $s_i$ has a unique dominant monomial $\tilde{
				s_i}=f_{i,\alpha}(z_{m+1},\ldots z_n) z_0^{
				\alpha_{0,i}}
			\ldots z_m^{\alpha_{m,i}}$ in its Taylor expansion, and any other term has modulus bounded by $Ce^{-\delta |\log |t||}|z_0|^{
				\alpha_{0,i}}
			\ldots |z_m|^{\alpha_{m,i}}$. Thus it suffices to prove the above inequality, but with each $
			s_i$ replaced by the leading monomial term $\tilde{s}_i$. We classify these leading monomials according to the exponents $(\alpha_{0,i},\ldots \alpha_{m,i})$ modulo $\Z(1,\ldots 1)$ (Notice here the multiplicity of $E_i$ is $b_i=1$, by the semistable SNC model assumption). If the exponents of two such monomials lie in different equivalence classes, then after eliminating $z_0$ by $z_0= t z_1^{-1}\ldots z_m^{-1}$, the monomials $z_1^{\alpha_{1,i}-\alpha_{0,i}}\ldots z_m^{\alpha_{m,i}-\alpha_{0,i}}$ are distinct. The crucial observation is that by the $L^2$-\emph{orthogonality of Fourier modes} with respect to the $T^m$-action, the integral 
			\[
			\int |\sum_1^{N(l)} a_i \tilde{s}_i|^2 \bigwedge_1^m \sqrt{-1} d\log z_i\wedge d\overline{\log z_i}\wedge \bigwedge_{m+1}^n \sqrt{-1} dz_j\wedge d\bar{z}_j  
			\]
			decouples into a sum of contributions from each equivalence class of exponents. On the other hand, by the equivalent characterisation of valuative independence in section \ref{sect:equivalentcharacterisation}, within each equivalence class of exponents, the functions $f_{i,\alpha}$ are linearly independent over $\C$, so for any $
			\tilde{a}_i\in \C$, 
			\[
			\int_{D(\delta/2)^{n-m}} |\sum \tilde{a}_i f_{i,\alpha}|^2 \bigwedge_{m+1}^n \sqrt{-1} dz_j\wedge d\bar{z}_j  \geq C(\delta)^{-1} \sum |\tilde{a}_i|^2.
			\]
			We emphasize that the constant is independent of $\tilde{a}_i$ and $t$.
			We apply this to $\tilde{a}_i=a_i t^{n_i}$ for appropriate $n_i$, chosen so that the exponents
			$(a_{0,i}-n_i,\ldots ,a_{m,i}-n_i)$ are equal for all $i$ in the given equivalence class.

			Combining the above, on the open subset $U_\delta$, we have
			\[
			\begin{split}
				& \int |\sum_1^{N(l)} a_i \tilde{s}_i|^2 \bigwedge_1^m \sqrt{-1} d\log z_i\wedge d\overline{\log z_i}\wedge \bigwedge_{m+1}^n \sqrt{-1} dz_j\wedge d\bar{z}_j  \\
				= &\sum_{equiv. class}  \int |\sum_i a_i \tilde{s}_i|^2 \bigwedge_1^m \sqrt{-1} d\log z_i\wedge d\overline{\log z_i}\wedge \bigwedge_{m+1}^n \sqrt{-1} dz_j\wedge d\bar{z}_j 
				\\
				\geq & C^{-1} \sum_{equiv. class} |z_0|^{
					2(\alpha_{0,i}-n_i)}
				\ldots |z_m|^{2(\alpha_{m,i}-n_i)} \int_{D(\delta/2)^{n-m}} |\sum_i \tilde{a}_i f_{i,\alpha}|^2  
				\\
				\geq & C^{-1} \sum_{equiv. class} \sum_i |\tilde{a}_i|^2 |z_0|^{
					2(\alpha_{0,i}-n_i)}
				\ldots |z_m|^{2(\alpha_{m,i}-n_i)}
				\\
				= &  C^{-1} \sum_1^{N(l)}  |a_i|^2 |z_0|^{
					2\alpha_{0,i}}
				\ldots |z_m|^{2\alpha_{m,i}} 
				\\
				\geq & C^{-1} \sum_1^{N(l)} |a_i|^2 |\tilde{s}_i|^2.
			\end{split}
			\]
			Here the constants change from line to line, and depend on $\delta$ but not on $a_i$ or $t$. The conclusion follows.
		\end{proof}

		We now define $\tilde{U}'_\delta\subset X_t$ to be the union of all the good dyadic scales in $\tilde{U}_\delta$ which are disjoint from $\mathcal{S}_\delta$. We define 
		\[
		U_\delta=  \{     \zeta\in \tilde{U}_\delta': | \phi- \bar{\phi} |\leq C_2(\delta) |\log |t||^{1/2 }                \} \subset \tilde{U}'_\delta\subset X_t.
		\]

		\begin{lem}\label{lem:Udeltameasure}
			Given small $\delta>0$, then for $t$ small enough depending on $\delta$, 
			the normalised CY measure satisfies the uniform estimate $\mu_t(X_t\setminus U_\delta)\leq C\delta$, where the constant $C$ is independent of $\delta$. 
		\end{lem}
		
		\begin{proof}
			The estimate $\mu_t( X_t\setminus \tilde{U}'_\delta ) \leq C\delta$ follows from three facts: $\mu_t(X\setminus \tilde{U}_\delta)\leq C\delta$, and $\mu_t(\mathcal{S}_\delta)\leq C\delta$, and the union of all bad dyadic scales have normalised CY measure bounded by $C\delta$.  On the other hand, the large deviation estimate of Cor. \ref{cor:largedeviation}
			implies $\mu_t(  \tilde{U}_\delta'\setminus U_\delta   )\leq C\delta$.
		\end{proof}

		\begin{cor}\label{cor:asymptoticorthogonality2}
			Given $l$ large enough, and small $\delta>0$, the following holds when $t$ is small enough depending on $l,\delta$. For any $a_1,\ldots a_{N(l)}\in \C$, and for any $\zeta$ in  $U_\delta$ whose $(x_i)$-coordinates lies in $O_k$, we have
			\[
			\begin{split}
				& \frac{1}{2l}\log 	\int_{X_t}  |\sum_i a_i s_i^{(k)}|_{h^{\otimes l}}^2  e^{-2l\phi}  \sqrt{-1}^{n^2}\Omega_t\wedge \overline{\Omega}_t 
				\\
				\geq  &  \max_{1\leq i\leq N(l)} (l^{-1}\log  |a_i| |s_i^{(k)}|_{h^{\otimes l}} - \phi ) (\zeta)- C(\delta,l)  |\log |t||^{1/2} .
			\end{split}
			\]
		\end{cor}

		\begin{proof}
			The constants $C$ will depend on $\delta,l$, but not on $\zeta\in U_\delta$ and $ t, a_i$. We omit the superscript $k$.
			By construction $\zeta\in U_\delta$ belongs to a good dyadic scale disjoint from $\mathcal{S}_\delta$. 
			By Lemma \ref{lem:asymptoticorthogonality1}, 
			\[
			\log 	\int   |\sum_i a_i s_i|_{h^{\otimes l}}^2  \sqrt{-1}^{n^2}\Omega_t\wedge \overline{\Omega}_t \geq  \log \sum_i  |a_i|^2 |s_i|_{h^{\otimes l}}^2  (\zeta) -C\geq  \log  \max_i  |a_i|^2 |s_i|_{h^{\otimes l}}^2(\zeta)- C,
			\]
			where the LHS integral is over the shrunken dyadic scale.

			By the Lipschitz estimate (\ref{lem:Lip}), within the dyadic scale $|\bar{\phi}- \bar{\phi}(\zeta)|\leq C$. By (\ref{eqn:subharmonic}), within the shrunken dyadic scale
			$\phi\leq \bar{\phi}(\zeta)+ C |\log |t||^{1/2}$. Since $\zeta\in U_\delta$, we have $|\phi(\zeta)-\bar{\phi}(\zeta)|\leq C  |\log |t||^{1/2}$. Combining the above,
			\[
			\begin{split}
				&	\frac{1}{2l}\log 	\int   |\sum_i a_i s_i|_{h^{\otimes l}}^2  e^{-2l\phi}    \sqrt{-1}^{n^2}    \Omega_t\wedge \overline{\Omega}_t 
				\\
				\geq &  \frac{1}{2l}\log 	\int   |\sum_i a_i s_i|_{h^{\otimes l}}^2   e^{-2l\phi(\zeta)}     \sqrt{-1}^{n^2}   
				\Omega_t\wedge \overline{\Omega}_t - C |\log |t||^{1/2}
				\\
				\geq &    \max_i  (l^{-1}\log  |a_i| |s_i|_{h^{\otimes l}} - \phi ) (\zeta) - C |\log |t||^{1/2},
			\end{split}
			\]
			where in the last step we absorbed some constant $C$ into the $|\log |t||^{1/2}$ term, for $t$ small enough depending on $\delta,l$. 
		\end{proof}

		\subsection{Fubini-Study type metric}

		We shall now introduce the Fubini-Study type metric which approximates the Calabi-Yau metric.

		Let $l\geq l_0$, and 
		we apply the constructions of Section \ref{sect:averagepotential}, \ref{sect:L2almostorthogonality} to the CY potential $\phi=\phi_{CY,t}$, and define the open subset $U_\delta$ etc. Let $S_1,\ldots S_N$ be the collection of meromorphic sections produced in Section \ref{sect:filtration}; in general $N$ may be bigger than $N(l)=\dim H^0(X_t, lL)$. 
		For each $\alpha=1,\ldots , N$, we define
		\begin{equation}
			c_\alpha^l(t)= \sup_{U_\delta} \frac{1}{|\log |t||}(-\phi_{CY,t}+l^{-1} \log |S_\alpha|_{h^{\otimes l}}),
		\end{equation}
		and the continuous semi-positive potential on $(X_t,L)$,
		\begin{equation}
			\phi_{l,t}= \max_{1\leq \alpha\leq N} (l^{-1} \log |S_\alpha|_{h^{\otimes l}} -|\log |t| |c_\alpha^l(t)). 
		\end{equation}

		\begin{prop}\label{prop:potentiallowerbound}
			(Potential lower bound) 
			For $t$ small enough depending on $\delta$, we have the uniform bound
			$
			\phi_{CY,t}- \phi_{l,t} \geq  -C |\log |t|| \delta^{1/2n},
			$
			where $C$ is independent of $\delta,l, t$.
		\end{prop}

		\begin{proof}
			By construction $\phi_{CY,t}\geq  \phi_{l,t}$ on $U_\delta$. We apply the pluripotential theoretic estimate \cite[Thm. 2.7]{LiFermat} to the background metric $\omega= \frac{1}{|\log |t||} ( \omega_h+ dd^c \phi_{l,t}  )$, using the uniform Skoda inequality in \cite{LiuniformSkoda}. We
			deduce that the potential of the CY metric $\frac{1}{|\log |t||} ( \omega_h+ dd^c \phi_{CY,t}  )$ in the class $\frac{1}{|\log |t||}c_1(L)$ satisfies the estimate
			\[
			\frac{1}{|\log |t||}  (\phi_{CY,t}- \phi_{l,t})  \geq  -C \mu_t(  \{    \phi_{CY,t}- \phi_{l,t}\leq 0     \}   )^{ 1/2n } \geq -C \mu_t(  X_t\setminus U_\delta  )^{ 1/2n }.
			\]
			But  $\mu_t(   X_t\setminus U_\delta  )\leq C\delta$  by Lemma \ref{lem:Udeltameasure}, hence the result.
		\end{proof}

		\begin{prop}(Potential upper bound)\label{prop:potentialupperbound}
			Let $l\geq l_0$ be large enough, and $\delta>0$ small. For $t$ small enough depending on $\delta,l$, we have the uniform estimate
			$
			\phi_{CY,t}- \phi_{l,t}  \leq 	  \frac{C}{l} |\log |t|| ,
			$
			where $C$ is independent of $l,\delta, t$. 
		\end{prop}

		\begin{proof}
			For any $z\in X_t$, we apply Prop. \ref{OhsawaTakegoshi} to find a section $s$,  such that 
			\[
			\frac{1}{2l}\log \int_{X_t} |s|^2_{h^{\otimes l }} e^{ -2l \phi_{CY,t} } \sqrt{-1}^{n^2}\Omega_t \wedge \overline{\Omega}_t \leq    ( \frac{1}{l} \log |s|_{h^{\otimes l}} -\phi_{CY,t}  )(z) + \frac{C}{l} |\log |t||.
			\]	
			We consider $\zeta\in U_\delta$ whose $(x_i)$-coordinates lies in a given connected component $O_k\subset  Sk(X)\setminus \mathcal{S}$. We write $s=\sum_1^{N(l)} a_i^k s_i^{(k)}$ for $a_i^k\in \C$. By Cor. \ref{cor:asymptoticorthogonality2}, 
			\[
			\begin{split}
				&	\frac{1}{2l}\log 	\int_{X_t}  |\sum_i a_i^k s_i^{(k)}|_{h^{\otimes l}}^2  e^{-2l\phi_{CY,t}}  \sqrt{-1}^{n^2}\Omega_t\wedge \overline{\Omega}_t 
				\\
				\geq &  \max_{1\leq i\leq N(l)}   \sup_{\zeta} \left(l^{-1}\log ( |a_i^k| |s_i^{(k)}|_{h^{\otimes l}}) - \phi_{CY,t}\right) (\zeta)- C(\delta,l)  |\log |t||^{1/2} 
				\\
				= &    \max_{1\leq i\leq N(l)} l^{-1} \log |a_i^k| +  \sup_{\zeta}  ( l^{-1} \log  |s_i^{(k)}|_{h^{\otimes l}} - \phi_{CY,t}) (\zeta)- C(\delta,l)  |\log |t||^{1/2} 
			\end{split}
			\]
			For small $t$ depending on $\delta, l$ we can absort the $ |\log |t||^{1/2}$ term into the $|\log |t||$ term, to obtain
			\begin{equation}\label{eqn:valuativeconstraint}
				\begin{split}
					&	( \frac{1}{l} \log |s|_{h^{\otimes l}} -\phi_{CY,t}  )(z) + \frac{C}{l} |\log |t|| 
					\\
					\geq  & \max_{1\leq i\leq N(l)}    (l^{-1}\log  |a_i^k| +   \sup_{\zeta}  ( l^{-1} \log  |s_i^{(k)}|_{h^{\otimes l}} - \phi_{CY,t}) (\zeta)  .
				\end{split}
			\end{equation}

			We view this key estimate as giving a collection of upper bounds on $|a_i^k|$, namely a prescription of \emph{valuative constraints} (\cf Section \ref{sect:valuativeconstraints})
			\[
			|a_i^k| \leq |t|^{\mu_{i,k} },
			\]
			for suitable real parameters $\mu_{i,k}\in \R$. By Cor. \ref{cor:fromfibretofamily}, we can replace the section $s$ by $c|t|^\nu S_\alpha$ for some $1\leq \alpha\leq N$, such that both the evaluation at $z$ and the valuative constraints are preserved:
			\begin{equation}
				\begin{cases}
					s(z)= c|t|^\nu S_\alpha(z), 
					\\
					c|t|^\nu S_\alpha= \sum_i (a_i^k)_{new} s_i^{(k)} ,\quad  |(a_i^k)_{new}| \leq C(l)|t|^{\mu_{i,k}}.
				\end{cases}
			\end{equation}
			Thus after absorbing $\log C(l)$ into the $\frac{C}{l} |\log |t||$ term, 
			\begin{equation*}
				\begin{split}
					&	( \frac{1}{l} \log |  c |t|^\nu S_\alpha   |_{h^{\otimes l}} -\phi_{CY,t}  )(z) + \frac{C}{l} |\log |t|| 
					\\
					\geq  & \max_{1\leq i\leq N(l)}    (l^{-1}\log  |(a_i^k)_{new}| +   \sup_{\zeta}  ( l^{-1} \log  |s_i^{(k)}|_{h^{\otimes l}} - \phi_{CY,t}) (\zeta)  
					\\
					= &      \sup_{\zeta}  ( l^{-1} \log \max_{1\leq i\leq N(l)}  | (a_i^k)_{new}s_i^{(k)}|_{h^{\otimes l}} - \phi_{CY,t}) (\zeta)  .
				\end{split}
			\end{equation*}

			But by the triangle inequality
			\[
			|c|t|^\nu S_\alpha|_{h^{\otimes l}}  \leq \sum_i  |(a_i^k)_{new} s_i^{(k) }| _{h^{\otimes l}} \leq N(l) \max_{1\leq i\leq N(l)}  | (a_i^k)_{new}s_i^{(k)}|_{h^{\otimes l}} ,
			\]
			hence after absorbing $\log N(l)$ into the $\frac{C}{l} |\log |t||$ term, we obtain
			\[
			( \frac{1}{l} \log | c |t|^\nu S_\alpha   |_{h^{\otimes l}} -\phi_{CY,t}  )(z) + \frac{C}{l} |\log |t||  \geq \sup_{\zeta}  ( l^{-1} \log    	|c |t|^\nu S_\alpha|_{h^{\otimes l}}         - \phi_{CY,t}) (\zeta) .
			\]
			Cancelling out the $l^{-1}\log |c|t|^\nu|$ term from both sides, 
			\[
			( \frac{1}{l} \log | S_\alpha   |_{h^{\otimes l}} -\phi_{CY,t}  )(z) + \frac{C}{l} |\log |t||  \geq \sup_{\zeta}  ( l^{-1} \log    	|S_\alpha|_{h^{\otimes l}}         - \phi_{CY,t}) (\zeta) .
			\]
			This estimate holds for all the connected components $O_k\subset Sk(X)\setminus \mathcal{S}$, whence by the definition of $c_\alpha^l$, we have
			\[
			( \frac{1}{l} \log |  S_\alpha   |_{h^{\otimes l}} -\phi_{CY,t}  )(z) + \frac{C}{l} |\log |t||  \geq \sup_{\zeta\in U_\delta}  ( l^{-1} \log    	| S_\alpha|_{h^{\otimes l}}         - \phi_{CY,t}) (\zeta) = |\log |t|| c_\alpha^l .
			\]
			Thus
			\[
			\phi_{CY,t}(z) \leq l^{-1} \log |S_\alpha|_{h^{\otimes l}}(z) - |\log |t|| c_\alpha^l(t)+   \frac{C}{l} |\log |t|| \leq \phi_{l,t}+   \frac{C}{l} |\log |t|| ,
			\]
			as required. 
		\end{proof}

		\begin{cor}\label{cor:uniformFS}
			(Uniform Fubini-Study approximation)
			Let $l$ be large enough, then for $t$ small enough depending on $l$, we have
			\begin{equation}
				|\phi_{CY,t}- \phi_{l,t} |\leq \frac{C}{l} |\log |t||.
			\end{equation}
		\end{cor}

		\begin{proof}
			It suffices to choose $\delta$ such that $\delta^{1/2n} \leq l^{-1}$, and let $t$ be small enough depending on $l,\delta$. Then Prop. \ref{prop:potentiallowerbound}, \ref{prop:potentialupperbound} imply the result. 
		\end{proof}

		\subsection{Finishing the proof of Thm. \ref{thm:hybridconvergence}}

		For any fixed large $l$, by the uniform $L^\infty$ estimate in Lemma \ref{lem:Linfty}, we see that $|c_\alpha^l(t)|$ is bounded independent of $t$. By a \emph{diagonal argument}, for every sequence of $t=t_k\to 0$, we can find some subsequence, such that for any $l\geq l_0$, and any $\alpha=1,\ldots N$, we have the subsequential convergence 
		$
		c_\alpha^l(t)\to c_\alpha^l(0),
		$
		for some $c_\alpha^l(0)\in \R$ depending on the subsequence.

		Applying Section \ref{NAMAcxMA} and Cor. \ref{cor:uniformFS}, we then deduce that the $C^0$-limit of $\frac{1}{|\log |t||}\phi_{CY,t}$ exists along this given subsequence, and defines a limiting function $\phi_0\in \text{CPSH}(X,\mathcal{L})$. In our normalisation convention $\sup_{X_t} \phi_{CY,t}=0$, so $\sup_{X_K^{an}} \phi_0=0$. 
		Moreover, the complex MA measures on $X_t$ associated with $\phi_{CY,t}$ converge to the NA MA measure of the limiting function $\phi_0\in \text{CPSH}(X,\mathcal{L})$. But the hybrid topology limit of the complex MA measure $\mu_t$ is the normalised Lebesgue measure $\mu_0$ supported on $Sk(X)\subset X_K^{an}$ by Thm. \ref*{Volumeconvergence}. 	Thus by Thm. \ref{thm:BFJ}, $\phi_0$ is the potential of the  \emph{unique} NA CY metric $\norm{\cdot}_{CY,0}$ with the normalisation convention $\sup_{X_K^{an}} \phi_0=0$. By this uniqueness, the limit $\phi_0 $ is in fact independent of the subsequence. This implies that $\frac{1}{|\log |t||}\phi_{CY,t}$ converge in the $C^0$ hybrid topology to $\phi_0$ on $X_K^{an}$ as $t\to 0$.

			\begin{rmk}
			The above proof does not a priori assume the existence of the NA CY metric, but rather deduces its existence by taking the potential theoretic limit of the complex geometric CY metrics. 
			
			We also note that the CY metric condition is only used in the potential lower bound Prop. \ref{prop:potentiallowerbound} and the uniform $L^\infty$ bound Lemma \ref{lem:Linfty}, and the same estimate works for the more general complex MA equations with RHS $f_t\mu_t$, as long as $\norm{f_t}_{L^\infty}\leq C$ independent of $t$. Furthermore, if we assume the weak convergence of measures $f_t\mu_t\to f_0 \mu_0$, then the same argument shows the hybrid topology $C^0$-convergence of the normalised K\"ahler potentials, to the Boucksom-Favre-Jonsson solution of the NA MA equation with RHS measure $f_0\mu_0$. 
		\end{rmk}

		\begin{rmk}
			Pille-Schneider \cite[Cor. 5.2]{PilleSchneider} has also used the uniqueness of the NA CY metric to deduce some convergence property for the CY potentials independent of the subsequence .
		\end{rmk}

		\section{Optimal transport interpretation}\label{sect:optimaltransport}

		The goal of this section is to give an interpretation of the NA MA solution in terms of the Kontorovich dual formulation of an optimal transport problem, assuming the existence of a canonical basis of sections with good properties.

		\subsection{Assumptions on the valuative independent basis}\label{sect:assumptions}

		We shall assume the following setting:
		
		\begin{itemize}
			\item The polarised degeneration family of Calabi-Yau manifolds $X\to \mathbb{D}^*$ is a large complex structure limit, namely $\dim Sk(X)=n$.
			
			\item There is a (not necessarily SNC)\footnote{The model $(\mathcal{X},\mathcal{L})$ from the Gross-Siebert construction is usually not SNC. In particular, this $(\mathcal{X},\mathcal{L})$ has no relation to the semistable SNC model in Section \ref{sect:potentialconvergence} denoted by the same notation.} model $(\mathcal{X}, \mathcal{L})\to S$ of the polarised degeneration family $(X,L)$, and we are given $N(l)=\dim H^0(X_t,l L)   $  regular sections $\theta_p^l\in H^0(\mathcal{X}, l\mathcal{L})$, whose restrictions to the central fibre are a basis of $H^0(\mathcal{X}_0,l\mathcal{L})$. We assume that for any $l\geq 1$, the sections $\theta_p^l$ satisfy the \emph{valuative independence condition}.
			
			\item There is an $n$-dimensional integral polyhedral complex $B$, such that the sections $\theta_p^l$ are indexed by the rational points $p\in B(l^{-1}\Z)$. Moreover, under the multiplication map  $H^0(\mathcal{X}, l\mathcal{L})\otimes H^0(\mathcal{X}, l'\mathcal{L})\to H^0(\mathcal{X}, (l+l')\mathcal{L})$, we have
			\[
			\theta_p^l \theta_{p'}^{l'}= \theta^{l+l'}_{ \frac{lp+l'p'}{l+l'} } \mod t H^0(\mathcal{X}, (l+l')\mathcal{L}),
			\]
			for $p,p'$ on the same face of the integral polyhedral complex. We notice that if $p\in B(l^{-1}\Z), p'\in B(l'^{-1}\Z)$, then $\frac{lp+l'p'}{l+l'} \in B((l+l')^{-1}\Z)$.

		\end{itemize}

		We comment that these assumptions are directly motivated by the properties of theta functions in the Gross-Siebert programme, except for the valuative independence property, which is unknown in most cases (\cf Section \ref{sect:theta}).

		\subsection{Cost function}

		Our next goal is to extract a cost function $c: Sk(X)\times B\to \R$ from the valuation of $\theta_p^l$ on the essential skeleton.

		Using the local trivialising sections of the model line bundle $\mathcal{L}\to \mathcal{X}$, we can regard $\theta_p^l$ as local holomorphic functions on $X_K^{an}$, so for each $x\in Sk(X)\subset \Delta_{\mathcal{X}}\subset X_K^{an}$, we can take the valuation $val_x(\theta_p^l)$. This defines the real valued functions $val(\theta_p^l): Sk(X)\to \R$. By assumption $\theta_p^l$ is a holomorphic section on $\mathcal{L}^{\otimes l}\to \mathcal{X}$, so $val(\theta_p^l)\geq 0$, and since $\theta_p^l$ does not vanish identically on the central fibre, we know $\min_{x\in \Delta_\mathcal{X}} val_x(\theta_p^l)=0$.

		\begin{lem}\cite[Prop. 3.12]{Boucksomsemipositive}\label{lem:convexity}
			For each $p, l$, the function $-val(\theta_p^l)$ is a convex function on each face of $Sk(X)\subset \Delta_{\mathcal{X}}$.  
		\end{lem}

		\begin{lem}
			We have a uniform $L^\infty$-bound $\sup_{Sk(X)} |\frac{1}{l}val(\theta_p^l) |\leq C$, and a uniform Lipschitz bound
			\[
			|\frac{1}{l}val(\theta_p^l)(x)- \frac{1}{l}val(\theta_p^l)(x')|\leq C|x-x'|
			\]
			for $x,x'$ on the same face of $Sk(X)$. Here the constant is independent of $l, p$.
		\end{lem}

		\begin{proof}
			This is a special case of Lemma \ref{lem:Lip}.
		\end{proof}

		We are interested in the limiting behaviour of $val_x(\theta_p^l)$ in the $l\to +\infty$ limit. This will be based on the following simple observation:

		\begin{lem}\label{lem:semigroup}
			For any $x\in Sk(X)$, $l,l'\geq 1$, and $p,p'$ on the same face of $Sk(X)$, we have 
			\[
			val_x(\theta_p^l)+ val_x(\theta_{p'}^{l'} ) \leq  val_x(  \theta^{l+l'}_{ \frac{lp+l'p'}{l+l'} }).
			\]
		\end{lem}

		\begin{proof}
			We expand the product as a linear combination,
			\[
			\theta_p^l \theta_{p'}^{l'}= \sum_{r\in B((l+l')^{-1}\Z) } \alpha_{p,p',r} \theta_r^{l+l'},
			\]
			so that by the valuative independence condition,
			\[
			val_x(\theta_p^l)+ val_x(\theta_{p'}^{l'} )= val_x(\theta_p^l \theta_{p'}^{l'})=\min_{ r\in B((l+l')^{-1}\Z) } (val_x(\theta_r^{l+l'}) +val( \alpha_{p,p',r} )).
			\]
			Now by the assumptions, the coefficient 
			\[
			\alpha_{p,p',r} \in \begin{cases}
				1+  t\C[\![t]\!],\quad r= \frac{lp+l'p'}{l+l'},
				\\
				t\C[\![t]\!],\quad\text{otherwise}.
			\end{cases}
			\]
			thus the valuation $val( \alpha_{p,p',     \frac{lp+l'p'}{l+l'}  } )=0$, whence $val_x(\theta_p^l)+ val_x(\theta_{p'}^{l'} )\leq val_x(\theta_{   \frac{lp+l'p'}{l+l'}  }^{l+l'})$ as required. 
		\end{proof}

		Given a top dimensional face $\Delta_J\subset B$, we can associate the semigroup $\Gamma_J=\bigcup_{l\geq 0} 
		\Delta_J(l^{-1}\Z)\times 
		\{l\}$, where the semigroup product of $(p,l)$ and $(p',l')$ is $( \frac{lp+l'p'}{l+l'}, l+l' )$. A real valued function on a semigroup $F:\Gamma\to \R$ is called \emph{subadditive}, if
		\[
		F(a)+ F(b) 
		\geq F(a+b).
		\]
		Lemma \ref{lem:semigroup} says precisely that for any $x
		\in Sk(X)$, the function $-val_x( \theta_p^l)$ is subadditive on $\Gamma_J$. Together with the uniform $L^\infty$ bound in Lemma \ref{lem:Lip}, we can apply \cite[Thm. 3.1]{WittNystrom}, to show that for any sequence $p_l\in B(l^{-1}\Z)$ converging to any given point $p$ in the interior of a top dimensional face of $B$, the \emph{asymptotic limit}
		\[
		-\lim_{l\to +\infty} l^{-1}val_x( \theta_{p_l}^l)
		\]
		exists, and defines a function $c(x,p)$ depending only on $x,p$, which must be \emph{convex} in $p$ on the face of $B$. By passing Lemma \ref{lem:convexity}, \ref{lem:Lip} to the $l\to +\infty$ limit, we know that $c(x,p)$ is bounded in $L^\infty$, uniformly Lipschitz in $x$, and convex in $x$ on each face of $Sk(X)$. From the subadditivity again, we have
		\begin{equation}\label{eqn:valsmallerthanc}
			- l^{-1}val_x(\theta_p^l)\geq c(x,p),\quad \forall l\geq 1, \quad x\in Sk(X),\quad  p\in B(l^{-1}\Z).
		\end{equation}
		We will refer to $c(x,p)$ as the \emph{cost function}, following the optimal transport terminology.

		\begin{rmk}
			As a caveat, we have only defined $c(x,p)$ for $p$ in the interior of the $n$-dimensional faces of $B$, and we do not know if $c(x,p)$ extends to a continuous function on $Sk(X)\times B$. This will cause some technical complication later.
		\end{rmk}

		\subsection{c-transform and variational problem}

		Given the cost function $c(x,p)$, we can define the \emph{$c$-transform} for bounded functions $\phi: Sk(X)\to \R$ and $\psi: B\to \R$:
		\begin{equation}\label{eqn:ctransform}
			\psi^c(x)= \sup_{p\in B}c(x,p) - \psi(p),\quad
			\phi^c(p)= \sup_{x\in Sk(X)} c(x,p) -\phi(x).
		\end{equation}
		This can be viewed as a generalisation of the Legendre transform, by replacing vector spaces with polyhedral complexes, and replacing the standard pairing between dual vector spaces with the cost function $c(x,p)$. It follows from formal properties that

		\begin{lem}\label{lem:Pc}
			For any $\phi\in L^\infty(Sk(X))$ and $\psi\in L^\infty(B)$, we have $(\phi^c)^c\leq \phi$ and $(\psi^c)^c \leq \psi$. For $\phi_1,\phi_2\in L^\infty(Sk(X))$, we have $\norm{\phi_1^c-\phi_2^c}_{L^\infty}\leq \norm{\phi_1-\phi_2}_{L^\infty}$, and likewise for functions in $L^\infty(B)$.

			Moreover, for $\phi, \psi$ in the image of the $c$-transform, the $c$-transform is involutive: $(\phi^c)^c=\phi$ and $(\psi^c)^c=\psi$. We shall denote this class of functions $\phi$ as $\mathcal{P}_c$.
		\end{lem}

		\begin{lem}\label{lem:uniformcontinuityPc}
			The functions $\phi\in \mathcal{P}_c$ have uniform Lipschitz continuity on $Sk(X)$.
		\end{lem}

		\begin{proof}
			This follows formally from the uniform Lipschitz estimate
			\[
			\sup_{p\in B} |c(x,p)- c(x',p)|\leq C|x-x'|\]
			for any $x,x'$ on the same face of $Sk(X)$.
		\end{proof}

		The essential skeleton $Sk(X)$ carries the normalised Lebesgue measure $\mu_0$, and $B$ carries a Lebesgue measure induced from its integral structure, which we can normalise to a probability measure $\nu_0$ by dividing out the total volume. This leads to the following \emph{variational problem}: among functions $\phi\in \mathcal{P}_c$, minimise the functional
		\begin{equation}\label{eqn:functionalFmu0}
			\mathcal{F}_{\mu_0}(\phi)=  \int_{Sk(X)} \phi d\mu_0+ \int_{B} \phi^c d\nu_0.
		\end{equation}
		We observe that for any constant $a\in \R$, then $(\phi+a)^c= \phi^c-a$, so $\mathcal{F}_{\mu_0}(\phi)= \mathcal{F}_{\mu_0}(\phi+a)$. This variational problem is known as the \emph{Kontorovich dual formulation of optimal transport} with cost function $c(x,p)$ between the measures $(Sk(X),\mu_0)$ and $(B,\nu_0)$. Some nice background discussions can be found in \cite{Hultgren2}. The following existence result is quite standard.

		\begin{lem}
			There exists a minimiser inside $\mathcal{P}_c$ for the functional $\mathcal{F}_{\mu_0}$.  
		\end{lem}

		\begin{proof}
			By Lemma \ref{lem:Pc}, 
			the functional $\mathcal{F}_{\mu_0}$ depends continuously on $\phi\in\mathcal{P}_c$ with respect to the $C^0$-norm.	By Lemma \ref{lem:uniformcontinuityPc}, up to an additive normalising constant, the functions $\phi\in \mathcal{P}_c$ have uniform $L^\infty$ and Lipschitz bounds, so by Arzela-Ascoli, any minimising sequence subsequentially converges to a minimiser.
		\end{proof}

		\subsection{Relation to NA MA equation}

		Using the model line bundle $\mathcal{L}\to \mathcal{X}$, we can identity all  continuous NA semipositive metrics with their potential functions $\phi\in \text{CPSH}(X_K^{an},\mathcal{L})$.

		\begin{lem}\cite[Lemm 8.4]{Boucksom}\label{lem:domination}
			(Domination principle) Let $\phi, \phi'\in \text{CPSH}(X_K^{an},\mathcal{L})$, and suppose the NA MA measure for $\phi$ is supported on $Sk(X)\subset X_K^{an}$, and $\phi\geq \phi'$ on $Sk(X)$, then $\phi\geq \phi'$ on $X_K^{an}$. 
		\end{lem}

		\begin{prop}\label{prop:maxextension}
			For any potential $\phi\in \mathcal{P}_c$ defined on $Sk(X)\subset X_K^{an}$, there exists a unique extension $\tilde{\phi}\in \text{CPSH}(X_K^{an},\mathcal{L})$ such that  $\tilde{\phi}|_{Sk(X)}=\phi$, and 
			\[
			\tilde{\phi}(x)= \max\{ v(x) :v\in \text{CPSH}(X_K^{an},\mathcal{L}), v|_{Sk(X)}\leq \phi|_{Sk(X)}\},\quad \forall x\in X_K^{an}.
			\]
		\end{prop}

		\begin{proof}
			Let $\phi\in \mathcal{P}_c$, and let $\phi^c: B\to \R$ be its $c$-transform. 
			The psh envelop construction \cite[section 8]{Boucksomsemipositive} provides some continuous potential $\tilde{\phi}\in \text{CPSH}(X_K^{an},\mathcal{L})$, such that 
			\[
			\tilde{\phi}(x)= \max\{ v(x) :v\in \text{CPSH}(X_K^{an},\mathcal{L}), v|_{Sk(X)}\leq \phi|_{Sk(X)}\}. 
			\]
			We need to show that $\tilde{\phi}(x)= \phi(x)$ for any $x\in Sk(X)$. Clearly $\tilde{\phi}(x)\leq \phi(x)$.

			Given any $\epsilon>0$ and $x\in Sk(X)$, we can find some $p$ in the interior of some top dimensional face of $B$, such that
			\[
			\phi(x) \leq c(x,p)-\phi^c(p)+ \epsilon. 
			\]
			For any sequence $p_l\in B(l^{-1}\Z)$ tending to $p$, we have $-l^{-1}val(\theta_{p_l}^l)\to c(\cdot, p)$ uniformly in $x\in Sk(X)$, since the convergence holds pointwise in $x\in Sk(X)$, and by Lemma \ref{lem:Lip} we have uniform Lipschitz estimates in $x$. We find some NA Fubini-Study potential $\phi_l$ on $(X_K^{an},L)$, such that 
			\[
			\phi_l = l^{-1} \log |\theta_p^l|- \phi^c(p_l)-\epsilon=-l^{-1}val(\theta_{p_l}^l)-\phi^c(p_l)-\epsilon \quad \text{on $Sk(X)$}.
			\]
			Using also that $\phi^c$ is continuous in the interior of the top dimensional faces of $B$, the following holds for large $l$:
			\[
			\phi_l(y) \leq c(y ,p)- \phi^c(p) \leq \phi(y),\quad \forall y\in Sk(X).
			\]
			Now $\phi_l$ is a continuous psh potential, and by domination $\tilde{\phi}(x)\geq \phi_l(x)$, whence
			\[
			\tilde{\phi}(x) \geq \lim_{l\to +\infty} \phi_l(x) = c(x,p)- \phi^c(p) -\epsilon \geq \phi(x)-2\epsilon.
			\]
			Since $x,\epsilon$ are arbitrary, this proves that $\tilde{\phi}(x)\geq \phi(x)$ for any $x\in Sk(X)$, as required.
		\end{proof}

		The next result concerns Fubini-Study norms and will be useful for computing relative volume.

		\begin{lem}\label{lem:FSdiagonal}
			Let $\phi\in \text{CPSH}(X_K^{an},\mathcal{L})$ be a continuous psh potential satisfying 
			\[
			\phi(x)= \max\{ v(x) :v\in \text{CPSH}(X_K^{an},\mathcal{L}), v|_{Sk(X)}\leq \phi|_{Sk(X)}\},\quad \forall x\in X_K^{an}.
			\]
			For $l$ large enough so that the $K$-vector space $V_l=H^0(X_K,lL)$ generates the line bundle $L^{\otimes l}\to X_K$, we define
			\begin{equation}\label{eqn:phicl}
				\phi^c_l(p)= \max_{x\in Sk(X)} (  -\frac{1}{l} val_x(\theta_p^l)- \phi(x)),\quad \forall p\in B(l^{-1}\Z).
			\end{equation}
			Then the NA norms $\norm{\cdot}_{l\phi}$ on $V_l$  (See section \ref{sect:relativevolume}) is given by 
			\[
			\norm{ \sum_{p\in B(l^{-1}\Z)} a_p \theta_p^l }_{l\phi}= \max_p |a_p| e^{l \phi^c_l(p) },\quad \forall a_p\in K.
			\]

		\end{lem}

		\begin{proof}
			Let $s=\sum a_p \theta_p^l\in V_l$, where $a_p\in K$. By definition
			\[
			\norm{s}_{l\phi}= \sup_{X_K^{an}} |s| e^{-l\phi},
			\]
			where $|s|$ is defined by the model line bundle $\mathcal{L}\to \mathcal{X}$. The valuative independence condition implies that
			\[
			|s|= \max_{p\in B(l^{-1}\Z) } |a_p| |\theta_p^l| \quad \text{on $Sk(X)$},
			\]
			so by the definition of $\phi^c_l(p)$, we have
			\[
			\norm{s}_{l\phi} \geq \max_{Sk(X) }   |s| e^{-l\phi} =\max_{Sk(X)} \max_p |a_p| |\theta_p^l| e^{-l\phi} =\max_{p\in B(l^{-1}\Z)  } |a_p| e^{ l\phi^c_l(p)}.
			\]

			We now show the reverse inequality. We define the Fubini-Study potential on $X_K^{an}$, 
			\[
			\phi_l=\max_{p\in B(l^{-1}\Z )} l^{-1} \log |\theta_p^l| - \phi^c_l(p),
			\]		
			so $\phi_l\leq \phi$ on $Sk(X)$, whence $\phi_l\leq \phi$ on $X_K^{an}$ by the domination property of $\phi$. Thus
			\[
			\norm{ \theta_p^l}_{l\phi} =  \sup_{X_K^{an}} |\theta_p^l|  e^{-l\phi}
			\leq \sup_{X_K^{an}} |\theta_p^l|  e^{-l\phi_l}\leq e^{l\phi_l^c(p)},
			\]
			so the ultrametric inequality implies
			\[
			\norm{s}_{l\phi} \leq \max_p |a_p| \norm{\theta_p^l}_{l\phi} =\max_p |a_p| e^{l\phi_l^c(p)},
			\]
			completing the proof of the reverse inequality. 
		\end{proof}

		\begin{lem}\label{lem:convergencephic}
			In the notation of the above lemma, $\sup_{p\in B} |\phi_l^c|\leq C$ for some uniform constant  for all large $l$. Moreover,  for any compact subset $K$ in the interior of the top dimensional face of $B$, we have the uniform convergence $\phi_l^c(p_l)\to \phi^c(p)$ whenever $p_l\in B(l^{-1}\Z)$ converges to $p\in K$.
		\end{lem}

		\begin{proof}
			The uniform $L^\infty$ bound $\sup_{p\in B} |\phi_l^c|\leq C$  follows from Lemma \ref{lem:Lip} and the boundedness of $\phi$. The convergence statement for $\phi_l^c(p_l)$ follows from the convergence $-l^{-1}val_x(\theta_{p_l}^l)\to c(x,p)$ which is uniform for $x\in Sk(X)$. 
			%	Since $- l^{-1}val_x(\theta_p^l)\geq c(x,p)$ for all $l,p,x$, we have
			%	\[
			%	\phi^c_l(p_l) =  \max_{x\in Sk(X)} (  -\frac{1}{l} val_x(\theta_{p_l}^l)- \phi(x))\geq \max_{x\in Sk(X)} (  c(x,p_l)- \phi(x))=\phi^c(p_l),
			%	\]
			%	so by the continuity of $\phi^c$ in the interior of the $n$-dimensional faces of $B$, 
			%	\[
			%	\liminf_{l\to +\infty}  \phi^c_l(p_l) \geq \phi^c(p).
			%		\]
			%		On the other hand, for any $p\in K$, we can find some $x\in Sk(X)$ such that $\phi^c(p)=c(x,p)-\phi(x)$. 
		\end{proof}

		Our next goal is to compute the Monge-Amp\`ere energy via the relative volume interpretation in Section \ref{sect:relativevolume} \`a l\`a Boucksom-Eriksson \cite{Boucksomnew1}. This is inspired by the Chebyshev transform of Witt-Nystr\"om \cite{WittNystrom}.

		\begin{prop}\label{prop:MAenergy}
			Let $\phi\in \text{CPSH}(X_K^{an},\mathcal{L})$ be a continuous psh potential satisfying
			\[
			\phi(x)= \max\{ v(x) :v\in \text{CPSH}(X_K^{an},\mathcal{L}), v|_{Sk(X)}\leq \phi|_{Sk(X)}\},\quad \forall x\in X_K^{an}.
			\]
			Then up to an additive normalisation constant, the Monge-Amp\`ere energy is
			\[
			E(\phi)= -(L^n) \int_B  \phi^c(p) d\nu_0.
			\]
			Consequently, the functional $F_{\mu_0}$ from (\ref{eqn:Fmu01}) is equal to $\mathcal{F}_{\mu_0}$ from (\ref{eqn:functionalFmu0}). 
		\end{prop}

		\begin{proof}
			Take any other potential $\psi\in \text{CPSH}(X_K^{an},\mathcal{L})$ with the same domination property. Then by Lemma \ref{lem:FSdiagonal}, the NA norms $\norm{\cdot}_{l\phi}$ versus $\norm{\cdot}_{l\psi}$ have a simultaneous orthogonal basis $\theta_p^l$, namely for any $s=\sum_{p\in B(l^{-1}\Z)} a_p \theta_p^l $, and $a_p\in K$, we have 
			\[
			\norm{s }_{l\phi}= \max_p |a_p| e^{l \phi^c_l(p) },\quad 
			\norm{s }_{l\psi}= \max_p |a_p| e^{l \psi^c_l(p) }.
			\]
			Thus  the relative volume between the two NA norms  is (see Section \ref{sect:relativevolume})
			\begin{equation}\label{eqn:relativevolume1}
				vol(  \norm{\cdot}_{l\phi}, \norm{\cdot}_{l\psi}  )= \log \prod_{  p\in B(l^{-1}\Z) } e^{-l \phi^c_l(p)+l \psi^c_l(p)} = l\sum_{  p\in B(l^{-1}\Z) }  ( - \phi^c_l(p)+\psi^c_l(p) ).
			\end{equation}

			Let $\tilde{\nu}_0$ be the Lebesgue measure on $B$ induced from its integral structure. Then the total number of points
			\[
			|B(l^{-1}\Z)| = l^n \int_B d\tilde{\nu}_0 +o(l^n),\quad l\to +\infty.
			\]
			But $B(l^{-1}\Z)$ are bijective to the basis $\theta_p^l$ of $H^0(X_t,lL)$ for small $t$, so by Riemann-Roch and Serre vanishing theorem,
			\[
			|B(l^{-1}\Z)|= h^0(X_t, lL )=  \frac{   l^n  ( L^n)}{n!} +O(l^{n-1}).
			\]
			Compairing the two asymptotes, the total measure is
			\[
			\int_B d\tilde{\nu}_0=  \frac{   ( L^n)}{n!} .
			\]
			By definition the normalised measure $\nu_0=  \frac{n!}{(L^n)} \tilde{\nu}_0$, so its total measure is one.

			By Lemma \ref{lem:convergencephic}, the Riemann sum converges as $l\to +\infty$,
			\[
			l^{-n} \sum_{  p\in B(l^{-1}\Z) }  \phi^c_l(p)\to \int_B  \phi^c(p)   d\tilde{\nu}_0 = \frac{   ( L^n)}{n!}  \int_B  \phi^c(p)   d\nu_0 .
			\]
			By Section \ref{sect:relativevolume} and (\ref{eqn:relativevolume1}), 
			\[
			\begin{split}
				E(\phi)-E(\psi)=& \lim_{l\to+\infty} \frac{n!}{  l^{n+1}  } 	vol(  \norm{\cdot}_{l\phi}, \norm{\cdot}_{l\psi}  )
				\\
				= & n!  \lim_{l\to+\infty} l^{-n} \sum_{  p\in B(l^{-1}\Z) }  (-\phi^c_l(p)+ \psi^c_l(p))
				\\
				=&   ( L^n) \int_B  (-\phi^c(p)+\psi^c(p))   d\nu_0 .
			\end{split}
			\]
			Thus up to  an additive constant, the Monge-Amp\`ere energy admits the formula
			\[
			E(\phi)=- ( L^n)\int_B  \phi^c(p) d\nu_0.
			\]
			Consequently,
			\[
			\begin{split}
				F_{\mu_0}(\phi)= & - \frac{1}{(L^n)} E(\phi) + \int_{Sk(X)} \phi d\mu_0
				\\
				=&  \int_B  \phi^c(p)   d\nu_0+  \int_{Sk(X)} \phi d\mu_0
				\\
				=& \mathcal{F}_{\mu_0}(\phi)
			\end{split}
			\]
			as required.
		\end{proof}

		We finally arrive at the main theorem characterising the NA CY potential in term of the Kontorovich dual formulation of optimal transport problem.

		\begin{thm}\label{thm:optimaltransportinterpretation}
			Let $\phi_0$ be the potential for the NA CY metric $\norm{\cdot}_{CY,0}$. Then its restriction $\phi_0|_{Sk(X)}\in \mathcal{P}_c$, and $\phi_0|_{Sk(X)}$ is the unique minimiser of the functional $\mathcal{F}_{\mu_0}$ up to an additive constant.

		\end{thm}

		\begin{proof}
			Since the NA MA measure for the NA CY metric is supported on $Sk(X)$, the domination property applies to the NA CY potential $\phi_0$ by Lemma \ref{lem:domination}. In particular $\phi_0$ is uniquely determined by its restriction to $Sk(X)$, and we shall abuse the notation $\phi_0$ to also denote its restriction to $Sk(X)$.

			First we claim $\phi_0\in \mathcal{P}_c$. We let $(\phi_0^c)^c \in \mathcal{P}_c$ be the double $c$-transform.  By Prop. \ref{prop:maxextension}, we can find the unique extension of $(\phi^c)^c$ to $\psi\in \text{CPSH}(X_K^{an},\mathcal{L})$ with the domination property.   Applying Prop. \ref{prop:MAenergy} to both $\phi_0$  and $\psi$, we deduce that 
			\[
			F_{\mu_0}( \phi_0) = \mathcal{F}_{\mu_0}(\phi_0) , \quad F_{\mu_0}( \psi) = \mathcal{F}_{\mu_0}(\psi). 
			\]
			Since $\phi_0^c=\psi^c$ but $\phi_0\geq \psi$ by Lemma \ref{lem:Pc}, we deduce 
			$F_{\mu_0}(\phi_0) \geq F_{\mu_0} (\psi)$. But the NA CY potential is characterised as the unique minimiser of $F_{\mu_0}$ up to an additive constant, so in fact $\phi_0= \psi+\text{const}$, whence $\phi_0\in \mathcal{P}_c$.

			Next we show $\phi_0$ is the unique minimiser of $\mathcal{F}_{\mu_0}$ up to additive constant. For this we only need to replace $\psi$ in the above argument by any $\phi\in \mathcal{P}_c$, and use its unique maximal extension $\phi\in \text{CPSH}(X_K^{an},\mathcal{L})$ as the competitor for the functional $F_{\mu_0}$. Again we have
			\[
			F_{\mu_0}( \phi_0) = \mathcal{F}_{\mu_0}(\phi_0) , \quad F_{\mu_0}( \phi) = \mathcal{F}_{\mu_0}(\phi). 
			\]
			The claim follows from the fact that $\phi_0$ is the unique minimiser for $F_{\mu_0}$ up to constant. 
		\end{proof}

		\subsection{Variant: Intermediate complex structure limit}

		We now explain how a variant of the above techniques work in the Example \ref{eg:intermediate} of intermediate complex structure limits.

		We start by revisiting the basis $\theta_i^l$ from Example \ref{eg:intermediate}. 
		The essential skeleton is identified with the simplex
		\[
		Sk(X)= \{     (x_0,\ldots x_m)\in \R_{\geq 0}^{m+1}:  \sum_0^m x_i=1     \}.
		\]
		We let
		\begin{equation}
			B=\cup_{k=0}^m \Delta_k^\vee,\quad \Delta_k^\vee= \{    (p_0,\ldots p_m)\in \R^{m+1}_{\leq 0}: p_k=0, \sum d_i p_i \geq -1        \}.
		\end{equation}
		For $l=1,2,\ldots$, the rational points in $B(l^{-1}\Z)$ are of the form $(-\frac{l_0}{l},\ldots - \frac{l_m}{l})$, where $l_i$ are integers satisfying (\ref{eqn:liexponents}).
		The basis $\{  \theta_i^l\}$ consists of sections
		\[
		F_0^{l_0}\ldots F_m^{l_m} \otimes \tau_a^{l-\sum d_il_i},
		\]
		where $\tau_a$ comes from a basis of $V_{l-\sum d_il_i}\subset H^0(M, (l-\sum l_id_i)L)$, with $\dim V_k= \dim H^0(E_J, kL)$. Thus each point $p\in B(l^{-1}\Z)$ corresponds to $\dim V_{l-\sum d_il_i}$ different basis sections. This \emph{multiplicity} is the key difference from the large complex structure limit case, and it arises because the growth rate of $h^0(X_t, lL)\sim  \frac{(L^n)}{n!}l^n$ is bigger than $|B(l^{-1}\Z)|=O(l^m)$ where $m=\dim B$.

		The valuation function restricted to $Sk(X)$ is simply 
		\[
		val_x(   	F_0^{l_0}\ldots F_m^{l_m} \otimes \tau_a^{l-\sum d_il_i}  )= val_x(   	F_0^{l_0}\ldots F_m^{l_m}  )= \sum_0^m l_i x_i,
		\]
		so using the index identification $(p_0,\ldots p_m)= (- \frac{l_0}{l},\ldots  - \frac{l_m}{l})$,
		\[
		- l^{-1}val_x(  F_0^{l_0}\ldots F_m^{l_m} \otimes \tau_a^{l-\sum d_il_i}     )=  -l^{-1}\sum_0^m l_i x_i =\sum_0^m x_i p_i= \langle x,p\rangle.
		\]
		By analogy, we define the \emph{cost function} as the $l\to +\infty$ limit of these valuation functions. In this simple case the sequence $\langle x,p\rangle$ is independent of $l$, so we set $c(x,p)= \langle x,p\rangle$. We can then define the $c$-transform as in (\ref{eqn:ctransform}), and define the function class $\mathcal{P}_c$ as before.

		The analogue of Prop. \ref{prop:MAenergy} takes the following form:

		\begin{prop}\label{prop:MAenergyintermediate}
			Let $\phi\in \text{CPSH}(X_K^{an},\mathcal{L})$ satisfy
			\[
			\phi(x)= \max\{ v(x) :v\in \text{CPSH}(X_K^{an},\mathcal{L}), v|_{Sk(X)}\leq \phi|_{Sk(X)}\},\quad \forall x\in X_K^{an}.
			\]
			Then up to an additive normalisation constant, the Monge-Amp\`ere energy is
			\[
			E(\phi)= -(L^n) \int_B  W(p)\phi^c(p) d\nu_0,\quad W(p)= (1+ \sum_0^m d_i p_i)^{n-m},
			\]
			and $d\nu_0$ is proportional to the Lebesgue measure on $B$ induced from its integral structure, with the normalisation determined by $\int_B W(p) d\nu_0=1$.
			Consequently, the functional $F_{\mu_0}$ from (\ref{eqn:Fmu01}) is equal to the functional
			\begin{equation}
				\mathcal{F}_{\mu_0}(\phi):=  \int_B  W(p)\phi^c(p) d\nu_0 + \int_{Sk(X)} \phi d\mu_0.
			\end{equation}
		\end{prop}

		\begin{proof}
			The key difference from Prop. \ref{prop:MAenergy} is that each $p\in B(l^{-1}\Z)$ corresponds to $h^0(E_J, (l-\sum d_il_i)L)$ basis sections $\theta_i^l$; this is where the weight factor $W(p)$ comes from.
			The analogue of (\ref{eqn:phicl}) is
			\[
			\max_{Sk(X)} (-l^{-1} val_x(    	F_0^{l_0}\ldots F_m^{l_m} \otimes \tau_a^{l-\sum d_il_i}  )-\phi(x))= \max_{Sk(X)} \langle x,p\rangle -\phi(x) = \phi^c(p).
			\]
			Notably this depends only on $p$, not on the basis section or $l$.

			Take any other potential $\psi\in \text{CPSH}(X_K^{an},\mathcal{L})$ with the same domination property. Then by Lemma \ref{lem:FSdiagonal}, the NA norms $\norm{\cdot}_{l\phi}$ versus $\norm{\cdot}_{l\psi}$ have a simultaneous orthogonal basis $\theta_i^l$, namely for any $s=\sum_1^{N(l)} a_i \theta_i^l $, and $a_i\in K$, we have 
			\[
			\norm{s }_{l\phi}=  \max_i |a_i| e^{l \phi^c(p) }  ,\quad 
			\norm{s }_{l\psi}= \max_i |a_i| e^{l \psi^c(p) }.
			\]
			Thus  the relative volume between the two NA norms  is (see Section \ref{sect:relativevolume})
			\begin{equation}\label{eqn:relativevolume2}
				\begin{split}
					& vol(  \norm{\cdot}_{l\phi}, \norm{\cdot}_{l\psi}  )= \log \prod_{ 1}^{ N(l) } e^{-l \phi^c(p)+ l \psi^c(p)} 
					\\
					= &l\sum_{  p\in B(l^{-1}\Z) } h^0(E_J, (l-\sum l_i d_i)L) ( - \phi^c(p)+\psi^c(p) ),
				\end{split}
			\end{equation}
			by computing in the simultaneous orthogonal basis $\{ \theta_i^l\}$.

			By Riemann-Roch and Serre vanishing,
			\[
			h^0(E_J, kL)=  \frac{  k^{n-m} (L^{n-m}|_ {E_J})   }   {  (n-m)! } +O(k^{n-m-1}).
			\]
			Thus
			\[
			\begin{split}
				h^0(   E_J, (l-\sum l_i d_i)L     )= & \frac{ (L^{n-m}|_{E_J}) (l-\sum d_il_i)^{n-m}    }   {  (n-m)! } +O(l^{n-m-1})
				\\
				=&   \frac{  (L^{n-m}|_{E_J}) }{ (n-m)!}  W(p)  l^{n-m}+ O(l^{n-m-1}).
			\end{split}
			\]
			Hence by (\ref{eqn:relativevolume2}),
			\[
			l^{-(n+1)}	vol(   \norm{\cdot}_{l\phi}, \norm{\cdot}_{l\psi}   )= \frac{
				(L^{n-m}|_{E_J} )}{ (n-m)!  }     l^{-m} (\sum_{  p\in B(l^{-1}\Z) } W(p) ( - \phi^c(p)+\psi^c(p) ) +O(l^{-1})).
			\]

			Let $\tilde{\nu}_0$ be the Lebesgue measure on $B$ induced from the integral structure. Then by Section \ref{sect:relativevolume} and the convergence of the above Riemann sums,
			\[
			\begin{split}
				E(\phi)-E(\psi)=&  \lim_{l\to +\infty} \frac{n! }{l^{n+1}} vol(   \norm{\cdot}_{l\phi}, \norm{\cdot}_{l\psi}   )
				\\
				=  &  \frac{n!
					(L^{n-m}|_{E_J} )}{ (n-m)!  } 
				\int_B W(p) ( - \phi^c(p)+\psi^c(p) )d\tilde{\nu}_0.
			\end{split}
			\]
			Up to the choice of an additive constant,
			\begin{equation}
				E(\phi)=  -\frac{n!  (L^{n-m}|_{E_J}) }{ (n-m)!}   \int_B W(p)   \phi^c(p)d\tilde{\nu}_0.
			\end{equation}
			It remains to determine the coefficient. A shortcut is to use the formula
			\[
			E(\phi+1)-E(\phi)=(L^n),
			\] 
			to see
			\[
			\frac{n!  (L^{n-m}|_{E_J}) }{ (n-m)!}   \int_B W(p)   d\tilde{\nu}_0=(L^n). 
			\]
			By definition $\nu_0= \frac{\tilde{\nu}_0}   {    \int_B W(p)d\tilde{\nu}_0 }$, so 
			\[
			E(\phi)=  - (L^n)  \int_B W(p)   \phi^c(p)d\nu_0.
			\]
			as required.
		\end{proof}

		Theorem \ref{thm:optimaltransportinterpretation} holds verbatim.

		\begin{thm}\label{thm:optimaltransportinterpretation2}
			Let $\phi_0$ be the potential for the NA CY metric $\norm{\cdot}_{CY,0}$. Then its restriction $\phi_0|_{Sk(X)}\in \mathcal{P}_c$, and $\phi_0|_{Sk(X)}$ is the unique minimiser of the functional $\mathcal{F}_{\mu_0}: \mathcal{P}_c\to \R$ up to an additive constant.

		\end{thm}

		\begin{rmk}
			The variational problem for $\mathcal{F}_{\mu_0}:\mathcal{P}_c\to \R$ is the Kontorovich dual formulation of the optimal transport problem in \cite[Section 4.2]{Liintermediate}, so we recover the main picture of \cite{Liintermediate}, that the solution of the optimal transport problem encodes both the NA CY metric, and the $t\to 0$ asymptote of the CY potentials on $X_t$. 
		\end{rmk}

		\begin{rmk}
			It is desirable to generalise Thm. \ref{thm:optimaltransportinterpretation2} to other polarised degeneration families of Calabi-Yau manifolds, in the non-large complex structure limit case $1\leq \dim Sk(X)\leq n-1$. It would be useful if the Gross-Siebert canonical basis construction can be extended to this setting. 
		\end{rmk}

		\section{Discussions}\label{sect:discussion}

		We now discuss some connections to the literature and some speculations.

		\subsection{Toric hypersurfaces in toric Fano manifolds}

		Let $\Delta\subset M_\R$ be an integral reflexive Delzant polytope, and $X_\Delta$ be the associated smooth toric Fano manifold of dimension $n+1$, with the anticanonical polarization $L\to X_\Delta$.
		The origin $0\in \Delta$ corresponds to a distinguished section $X_{can}\in H^0(X_\Delta, -K_{X_\Delta})$, which defines the toric boundary of $X_\Delta$. Let $F\in H^0(X_\Delta, -K_{X_{\Delta}})$ be a generic section, such that the divisor $\{ F=0 \}\subset X_\Delta$ is smooth, and intersects all the toric boundary strata transversely, and in particular does not pass through the finite number of toric fixed points on $X_\Delta$. We will consider the family of Calabi-Yau hypersurfaces as $t\to 0$:
		\begin{equation}\label{torichypersurface}
			X_t= \{ X_{can}+tF=0  \}\subset X_\Delta\times \C^*_t.
		\end{equation}
		Algebro-geometrically $X_t$ degenerates to the toric boundary.

		Hultgren et al. \cite{Hultgren}\cite{Hultgren2}\cite{LiFano} introduced the following Kontorovich dual formulation of optimal transport problem, as a tool to study the potential theoretic limit of the Calabi-Yau manifolds. Let $\Delta^\vee\subset N_\R= (M_\R)^*$ be the dual polytope of $\Delta$, so there is a natural pairing $\langle x,p\rangle $ between $x\in \partial\Delta^\vee$ and $p\in \partial\Delta$. Both $\partial \Delta$ and $\partial \Delta^\vee$ have a natural Lebesgue measure $dx$ and $dp$ induced by the integral polytope structure. The $c$-transform for any $\phi\in C^0(\partial \Delta)$ and $\psi\in C^0(\partial \Delta^\vee)$ is defined as
		\[
		\psi^*(x)= \sup_{p\in \partial\Delta}\langle x,p\rangle - \psi(p),\quad
		\phi^*(p)= \sup_{x\in \partial\Delta^\vee} \langle x, p\rangle -\phi(x).
		\]
		On the image of the $c$-transform, we have $(\phi^*)^*=\phi$ and $(\psi^*)^*=\psi$, and we denote this class of convex functions $\phi$ as $\mathcal{P}\subset C^0(\partial \Delta)$. The Kontorovich functional on $\mathcal{P}$ is 
		\begin{equation}\label{eqn:toricfunctional}
			\mathcal{F}(\phi)=  \frac{1}{\int_{\partial \Delta^\vee} dx}\int_{\partial \Delta_\lambda^\vee} \phi dx+ \frac{1}{\int_{\partial \Delta} dp}\ \int_{\partial \Delta} \phi^* dp.
		\end{equation}
		This functional is unchanged if one adds a constant to $\phi$. 
		Hultgren \cite[Thm 5.2]{Hultgren} observed that there is a \emph{unique minimiser} $\phi\in \mathcal{P}$ up to additive constant.

		\begin{rmk}
			This setting has a number of variants: one can relax the ambient toric Fano manifold to be a Fano orbifold, the polarisation does not need to be the anticanonical polarisation, and the similar formulation makes sense for complete intersections in Fano ambient manifolds \cite{Goto}. 	
		\end{rmk}

		The relation to geometry is as follows. Here $\partial \Delta^\vee$ is identified with the essential skeleton of the large complex structure limit of Calabi-Yau hypersurfaces $X_t$. Any $\phi\in \mathcal{P}$ extends canonically to a continuous convex function on 
		$M_\R$, which
		specifies a toric NA semipositive metric on $(X_\Delta^{an}, L)$, hence induces an ansatz metric on $(X_K^{an},L)$ by restriction. The hope is that this agrees with the NA CY metric. This would hold if the unique minimiser $\phi$ satisfies the \emph{real Monge-Amp\`ere equation} on the interior of every $n$-dimensional face of $\partial \Delta$ \cite{Hultgren2}\cite{Liintermediate}. The main difficulty to prove the real Monge-Amp\`ere equation is that both $\partial \Delta$ and $\partial \Delta^\vee$ have many faces, and we need more precise control on where the gradient $\nabla \phi$ lands.

		Two kinds of contrasting results are known:
		
		\begin{itemize}
			\item In a few examples with large discrete symmetry \cite{Hultgren}, and more generally when $\Delta, \Delta^\vee$ satisfy some explicit combinatorial condition \cite{LiFano}, then one can prove the real Monge-Amp\`ere equation for the minimiser $\phi$, and thereby deduce that $\phi$ agrees with the potential of the NA MA metric.

			\item  Hultgren-Andreasson \cite{Hultgren2} proved that the real Monge-Amp\`ere equation for the minimiser $\phi$ is equivalent to the existence of an optimal transport plan with some support conditions.	In many numerical examples \cite{Hultgren2}, it is found that such an optimal transport plan \emph{does not exist}. In other words, this strategy does not always work.

		\end{itemize}

		The results of this paper suggest a conceptual explanation why the above optimal transport problem only works in some cases but not others:

		\begin{itemize}
			\item  %The K\"ahler potentials corresponding to convex functions in $\mathcal{P}$ must extend to toric potentials on the ambient Fano manifold $X_\Delta$. %This class of potentials is strictly smaller than the set of all K\"ahler potentials on $(X_t,L)$, so there is no a priori necessity for the potential of the CY metrics on $(X_t,L)$ to be well approximated by such toric potentials.
			The class of all NA semipositive metrics on $(X_K^{an},L)$ is  \emph{larger} than the subset obtained by restricting semi-positive toric potentials on $X_\Delta$.  Thus there is no a priori necessity for the NA CY metric to correspond to some potential in $\mathcal{P}$; this good situation only sometimes occur.

			\item  Under the assumptions in section \ref{sect:assumptions}, then  by Thm. \ref{thm:optimaltransportinterpretation}, the unique minimiser to the variational problem $\mathcal{F}_{\mu_0}:\mathcal{P}_c\to \R$ encodes the NA CY metric.  It is plausible that the cost function $c(x,p)$ in this `correct' variational problem does not coincide with the cost function $\langle x,p\rangle$ in the `na\"ive' variational problem (\ref{eqn:toricfunctional}); rather $\langle x,p\rangle$ is only the first approximation to $c(x,p)$, which would explain why the minimisers of the two variational problems only agree in some cases.

			\item The cost function of the variational problem (\ref{eqn:toricfunctional}) relies on a \emph{fixed embedding} of the polarised degeneration family of CY manifolds in  some ambient Fano manifold. On the other hand, the cost function in $\mathcal{F}_{\mu_0}:\mathcal{P}_c\to \R$ relies on the \emph{asymptotic} information about the projective embeddings for high powers of $L\to X$.

		\end{itemize}

		\subsection{Metric SYZ conjecture}

		One of the main motivations for developing potential convergence and the optimal transport interpretation, is to attack the metric version of the SYZ conjecture \cite{SYZ}:

		\begin{conj}
			Given a polarised degeneration  family of $n$-dimensional  Calabi-Yau manifolds $(X_t, \omega_{CY,t})$ near the large complex structure limit, then there exist special Lagrangian $T^n$-fibrations on some subset $U_t\subset X_t$ for $0<|t|\ll 1$, such that the normalised CY measure $\mu_t(U_t)\to 1$ as $t\to 0$. 
		\end{conj}

		It is by now well understood that the key statement to prove is that in the generic region $U_t$, the CY metric should admit a local semi-flat metric asymptote as $t\to 0$,
		\[
		\omega_{CY,t}\sim  |\log |t|| dd^c u(  \frac{\log |z_1|}{\log |t|}, \ldots         \frac{\log |z_n|}{\log |t|}      ),
		\]
		where $u(x_1,\ldots x_n)$ is a a smooth convex function solving the real Monge-Amp\`ere equation $\det(D^2 u)=\text{const}$. The essential step is to produce a solution of this real MA equation, and justify that the local K\"ahler potential of the CY metric is well approximated by $|\log |t|| u(  \frac{\log |z_1|}{\log |t|}, \ldots         \frac{\log |z_n|}{\log |t|}      )$.

		This paper suggests the following strategy to attack the SYZ conjecture:
		
		\begin{itemize}
			\item  First, one needs to check the valuative independence condition for a sufficiently large collection of interesting examples. This is an unsolved problem in algebraic/tropical geometry.

			\item Next, one needs to compute the valuation of $\theta^l_p$ on $Sk(X)$, and extract the cost function $c(x,p)$ from the asymptotic limit of the valuation functions. One would like to understand $c(x,p)$ as explicitly as possible.

			\item By Thm. \ref{thm:hybridconvergence} and  Thm. \ref{thm:optimaltransportinterpretation}, the $C^0$ potential theoretic limit then boils down to the minimiser $\phi_0$ of the Kontorovich dual formulation of the optimal transport problem. 
			
			%	If one can accomplish the first two steps, then one can efficiently simulate the solution $\phi_0$ numerically, and whether $\phi_0$ locally solves the real MA equation can be viewed as a numerical test on the SYZ conjecture.

			\item It remains to study the regularity of the minimiser $\phi_0$, and prove that on an open subset of full Lebesgue measure in $Sk(X)$, the local convex function $\phi_0$ in fact solves the real MA equation. This $\phi_0$ would then play the role of $u$ above.
			
			It should be noted that for classical optimal transport problems between Euclidean domains with the Lebesgue measure, the optimiser solves the real MA equation when the cost function is the bilinear pairing $\langle x,p\rangle$. The new difficulty here is that $Sk(X)$ and $B$ are polyhedral complexes, and the cost function $c(x,p)$ may be complicated. At an intuitive level, it would be helpful if $c(x,p)$ behaves like a bilinear pairing on a large subset of $Sk(X)\times B$.

		\end{itemize}

		\subsection{Metric mirror duality}

		We now make some more speculative remarks about the best hope for a \emph{metric version of mirror symmetry}.

		Recall from Section \ref{sect:theta} that the Gross-Siebert theta function basis on $L\to X$ are parametrised by suitable rational points on the essential skeleton $B=Sk(X^\vee)$ of the mirror family $X^\vee$. In the best cases, both $X$ and $X^\vee$ are polarised large complex structure limits of Calabi-Yau manifolds, and they are \emph{mutually mirror} to each other. In such cases, there will also be a theta function basis on the tensor powers of $L^\vee\to X^\vee$, parametrised by the suitable rational points on $Sk(X)$.

		Suppose that the valuative independence condition is satisfied for both theta basis on $X$ (resp. on $X^\vee$). Then from the asymptotic information about the valuations of the theta function, we can extract the cost function $c(x,p)$ on $Sk(X)\times Sk(X^\vee)$, and another cost function $c^\vee(p,x)$ on $Sk(X^\vee)\times Sk(X)$ upon reversing the roles of $X$ and $X^\vee$.

		\begin{Question}
			Do we have the \emph{duality} property $c(x,p)= c^\vee(p,x)$ for any $x\in SK(X)$ and any $p\in Sk(X^\vee)$?
		\end{Question}

		\begin{rmk}
			This question has close analogy to the `theta reciprocity' in the context of cluster algebras, anounced in the Banff talk `Reciprocity for Valuations of Theta Functions' by Greg Muller.
		\end{rmk}

		This duality would have the following appealing consequence. The potential theoretic limit of the CY metrics on $L\to X_t$ is encoded into the unique minimiser $\phi_0\in \mathcal{P}_c$ for the Kontorovich functional
		\[
		\mathcal{F}_{\mu_0}(\phi)= \int_{Sk(X)} \phi d\mu_0+ \int_{Sk(X^\vee)}  \phi^c d\nu_0,
		\]
		where $d\mu_0$ and $d\nu_0$ are the normalised Lebesgue measures on $Sk(X)$ and $Sk(X^\vee)$. By the same reason, the potential theoretic limit of the CY metric on $L^\vee\to X^\vee$ is encoded into the unique minimiser $\psi_0\in \mathcal{P}_c^\vee$ for the Kontorovich functional
		\[
		\mathcal{F}^\vee_{\nu_0}(\psi): = \int_{Sk(X^\vee)} \psi d\nu_0+ \int_{Sk(X)}  \psi^c d\mu_0.
		\]
		Assuming that $c(x,p)=c^\vee(p,x)$, then we have the Legendre duality,
		\[
		\phi\mapsto \phi^c,  \quad  \psi\mapsto \psi^c.
		\]
		which sets up a canonical bijection between $\mathcal{P}_c$ and $\mathcal{P}_c^\vee$ (\cf Lem. \ref{lem:Pc}). The main observation is that the two Kontorovich functionals are identified by
		\[
		\mathcal{F}_{\mu_0}(\phi)=  \mathcal{F}_{\nu_0}^\vee(\phi^c), \quad \forall \phi\in \mathcal{P}_c.
		\]
		In particular, the two minimisers are related by $c$-transform,
		\[
		\phi_0=\psi_0^c,\quad \psi_0= \phi_0^c.
		\]
		This gives a sense that \emph{the potential theoretic limits for the CY metrics on the two sides of the mirror are related by Legendre duality}.

	\end{document}